\documentclass[12pt]{article}
\usepackage{graphicx}
\usepackage{amsmath,amsthm,amssymb,amsfonts,euscript,enumerate}
\usepackage{amsthm}
\usepackage{color,wrapfig}
\usepackage{pxfonts}
\usepackage{verbatim}
\newtheorem{thm}{Theorem}[section]
\newtheorem{cor}[thm]{Corollary}
\newtheorem{lem}[thm]{Lemma}

\newtheorem{rmk}[thm]{Remark}

\newcommand{\Rr}{{\mathbb{R}}}
\newcommand{\Z}{{\mathbb{Z}}}

\newcommand{\cD}{{\mathcal D}}
\newcommand{\cX}{{\mathcal X}}
\newcommand{\ve}{\varepsilon}

\newcommand{\3}{\varepsilon}
\newcommand{\4}{\widetilde}

\begin{document}
\title{A new proof for the existence of rotationally symmetric gradient Ricci solitons}
\author{Shu-Yu Hsu\\
Department of Mathematics\\
National Chung Cheng University\\
168 University Road, Min-Hsiung\\
Chia-Yi 621, Taiwan, R.O.C.\\
e-mail: shuyu.sy@gmail.com}
\date{April 6, 2024}
\smallbreak \maketitle
\begin{abstract}
We give a new proof for the existence of  rotationally symmetric steady and expanding gradient Ricci solitons in dimension $n+1$, $2\le n\le 4$,  with metric $g=\frac{da^2}{h(a^2)}+a^2d\,\sigma$ for some function $h$ where $d\sigma$ is the standard metric on the unit sphere $S^n$ in $\mathbb{R}^n$. More precisely for any $\lambda\ge 0$, $2\le n\le 4$ and $\mu_1\in\mathbb{R}$, we prove the existence of unique solution $h\in C^2((0,\infty))\cap C^1([0,\infty))$ for the equation $2r^2h(r)h_{rr}(r)=(n-1)h(r)(h(r)-1)+rh_r(r)(rh_r(r)-\lambda r-(n-1))$, $h(r)>0$, in $(0,\infty)$ satisfying $h(0)=1$, $h_r(0)=\mu_1$.
We also prove the existence of unique analytic solution of the about equation on $[0,\infty)$ for any $\lambda\ge 0$, $n\ge 2$ and $\mu_1\in\mathbb{R}$. Moreover we will  prove the asymptotic behaviour of the solution $h$ for any $n\ge 2$, $\lambda\ge 0$  and $\mu_1\in\mathbb{R}\setminus\{0\}$.
\end{abstract}

\vskip 0.2truein



\vskip 0.2truein
\setcounter{section}{0}

\section{Introduction}
\setcounter{equation}{0}
\setcounter{thm}{0}

Recently there is a lot of study on Ricci solitons by S.~Brendle,  R.L.~Bryant,  H.D.~Cao,   D.~Zhou,  M.~Feldman, T.~Ilmanen and D.~Knopf, Y.~Li and B.~Wang,  O.~Munteanu and N.~Sesum, P.~Petersen and W.~Wylie, \cite{B1}, \cite{Br},
\cite{C}, \cite{CZ}, \cite{FIK}, \cite{LW}, \cite{MS}, \cite{PW}, etc. Ricci solitons are used extensively in the study of Ricci flow. For example in \cite{B2} S.~Brendle used singular rotationally symmetric steady solitons to construct barrier functions which are used to prove a conjecture of Perelman on $3$-dimensional ancient $\kappa$ solution to the Ricci flow.

A Riemannian metric $g=(g_{ij})$ on a Riemannian manifold $M$  is said to be a gradient Ricci soliton if there exists a smooth function $f$ on $M$ and a constant $\lambda\in\mathbb{R}$ such that the Ricci tensor $R_{ij}$ of the metric $g$ satisfies
\begin{equation}\label{gradient-soliton-eqn}
R_{ij}=\nabla_i\nabla_jf-\lambda g_{ij}\quad\mbox{ on }M.
\end{equation}
The gradient soliton is called an expanding gradient Ricci soliton if $\lambda>0$. It is called a steady gradient  Ricci soliton if $\lambda=0$ and it is called a shrinking gradient  Ricci soliton if $\lambda<0$.

By an argument similar to R.L.~Bryant \cite{Br} for any $n\ge2$, if $(M,g)$ is a $(n+1)$-dimensional rotational symmetric gradient Ricci soliton which satisfies \eqref{gradient-soliton-eqn} for some smooth function $f$ and constant $\lambda\in\mathbb{R}$, then 
 \begin{equation}\label{g-rotational-form}
g=dt^2+a(t)^2d\sigma
\end{equation}
where $d\sigma$ is the standard metric on the unit sphere $S^n$ in $\mathbb{R}^n$, $a(t)$ satisfies
\begin{equation}\label{a-eqn}
a(t)^2a_t(t)a_{ttt}(t)=a(t)^2a_{tt}(t)^2-(n-1)a(t)a_{tt}(t)-\lambda a(t)^3a_{tt}(t)+a(t)a_t(t)^2a_{tt}(t)-(n-1)a_t(t)^2+(n-1)a_t(t)^4,
\end{equation}
and $f$ satisfies
\begin{equation*}
\left\{\begin{aligned}
a(t)f_{tt}(t)=&\lambda a(t)-na_{tt}(t)^2\\
a(t)a_{tt}(t)=&(n-1)(1-a_t(t)^2)-a(t)a_{t}(t)f_t(t)+\lambda a(t)^2.
\end{aligned}\right.
\end{equation*} 
As in \cite{Br}, by \eqref{a-eqn} if $a(0)=0$, one can choose $a_t(0)=1$. If one writes
\begin{equation}\label{g=h-form}
g=\frac{da^2}{h(a^2)}+a^2d\,\sigma
\end{equation}
where $h(r)$, $r=a^2$, and $a(t)$ satisfies
\begin{equation}\label{a-t-relation}
a_t(t)=\sqrt{h(a(t)^2)},
\end{equation}
then by \eqref{a-eqn} $h$ satisfies the singular elliptic equation
\begin{equation}\label{h-eqn}
2r^2h(r)h_{rr}(r)=(n-1)h(r)(h(r)-1)+rh_r(r)(rh_r(r)-\lambda r-(n-1)),\quad h(r)>0.
\end{equation}
Note that the radial sectional curvature and the orbital sectional curvature of $g$ is $-h_r(r)$ and   $(1-h(r))/r$ respectively. Hence the existence of gradient Ricci soliton is equivalent to the existence of solution of \eqref{h-eqn} and the positivity of the sectional curvatures of $g$ is equivalent to the condition that $h_r(r)<0$ and $0<h(r)<1$ holds for any $r>0$. 

In the paper \cite{Br} R.L.~Bryant by using power series method on \eqref{h-eqn} and phase plane analysis on the functions
\begin{equation}
x=a(t)^2,\quad y=a_t(t)^2, \quad z=a(t)a_{tt}(t),
\end{equation}
gave an incomplete proof of the existence of $3$-dimensional rotationally symmetric steady and expanding gradient  Ricci solitons. In this paper we will use fixed point argument to prove the existence of unique solution $h\in C^2((0,\infty))\cap C^1([0,\infty))$ of 
\begin{equation}\label{h-ode-initial-value-problem}
\left\{\begin{aligned}
&2r^2h(r)h_{rr}(r)=(n-1)h(r)(h(r)-1)+rh_r(r)(rh_r(r)-\lambda r-(n-1)),\quad h(r)>0, \quad\forall r>0\\
&h(0)=1,\quad h_r(0)=\mu_1
\end{aligned}\right.
\end{equation}
for any constant $\lambda\ge 0$, $\mu_1\in\mathbb{R}$   and integer $2\le n\le 4$. We will also use a modification of the power series method in \cite{Br} to proof the existence of unique analytic solutions of \eqref{h-ode-initial-value-problem} for any $n\ge 2$, $\lambda\ge 0$ and $\mu_1\in\Rr$. 
 
In \cite{Br} R.L.~Bryant by constructing explicit supersolutions of \eqref{h-eqn} proved that when  $\lambda>0$, $0<\mu_1<3\lambda/7$ and $n=2$, any solution $h$ of \eqref{h-eqn} which  satisfies the initial condition,
\begin{equation}\label{h-initial-condition}
h(0)=1\quad\mbox{ and }\quad h_r(0)=\mu_1,
\end{equation}
satisfies $h_r(r)>0$ for any $r\ge 0$ and $\underset{\substack{{r\to\infty}}}{\lim}\,h(r)\in (1,\infty)$ exists. However it is difficult to extend the method of \cite{Br} to the case $n\ge 2$ and $0<\mu_1<\lambda/n$. In this paper we will extend this result of \cite{Br} and prove the asymptotic behaviour of the solution $h$ of \eqref{h-ode-initial-value-problem} for any $n\ge 2$, $\lambda\ge 0$ and $\mu_1\in\Rr\setminus\{0\}$. The main results we obtain in this paper are the following.

\begin{thm}\label{h-existence-thm}
Let $\lambda\ge 0$, $\mu_1\in\mathbb{R}$ and $2\le n\le 4$. There exists a unique solution $h\in C^2((0,\infty))\cap C^1([0,\infty))$ of \eqref{h-ode-initial-value-problem}.
\end{thm} 

\begin{thm}\label{h-steady-soln-property-thm1}
Let $\lambda\ge 0$, $\mu_1\in\mathbb{R}$, $n\ge 2$ and $h\in C^2((0,\infty))\cap C^1([0,\infty))$ be  a solution  of
\eqref{h-ode-initial-value-problem}. Then 
\begin{equation}\label{hr-sign}
\left\{\begin{aligned}
&h(r)>1\quad\mbox{ and }\quad h_r(r)>0\quad\forall r>0\qquad\quad\mbox{ if }\mu_1>0\\
&0<h(r)<1\quad\mbox{ and }\quad h_r(r)<0\quad\forall r>0\quad\,\,\mbox{ if }\mu_1<0
\end{aligned}\right.
\end{equation}
\end{thm} 

\begin{thm}\label{h-lambda-neg-existence-thm}
Let $\lambda<0$, $2\le n\le 4$ and $\mu_1\in\mathbb{R}$. Then there exists a constant $L\ge -(n-1)/\lambda$ and a unique solution $h\in C^2((0,L))\cap C^1([0,L))$ of \eqref{h-eqn} in $(0,L)$ which satisfies \eqref{h-initial-condition}.
\end{thm} 

\begin{thm}\label{h-existence-analytic-soln-thm}
Let $\lambda\ge 0$, $\mu_1\in\mathbb{R}$ and $n\ge 2$. Then there exists a unique analytic solution $h$ of \eqref{h-ode-initial-value-problem} on $[0,\infty)$. 
\end{thm} 

\begin{thm}\label{steady-soln-asymptotic-behaviour-thm}
Let $\mu_1<0$ and $n\ge 2$.  Suppose $h\in C^2((0,\infty))\cap C^1([0,\infty))$ is  a solution  of
\begin{equation}\label{h-steady-soliton-eqn}
\left\{\begin{aligned}
&2r^2h(r)h_{rr}(r)=(n-1)h(r)(h(r)-1)+rh_r(r)(rh_r(r)-(n-1)),\quad h(r)>0, \quad\forall r>0\\
&h(0)=1,\quad h_r(0)=\mu_1.
\end{aligned}\right.
\end{equation} 
Then 
\begin{equation}\label{h-infty=0}
\lim_{r\to\infty}h(r)=0.
\end{equation}
\end{thm}

\begin{thm}\label{h-steady-soln-asymptotic-behaviour2-thm}
Let $\mu_1<0$ and $n\ge 2$. Suppose $h\in C^2((0,\infty))\cap C^1([0,\infty))$ is  a solution  of \eqref{h-steady-soliton-eqn}.  Then 
\begin{equation}\label{q-limit=-1}
\lim_{r\to\infty}\frac{rh_r(r)}{h(r)}=-1.
\end{equation}
\end{thm}

\begin{thm}\label{steady-soln-asymptotic-hr-behaviour-thm}
Let $\mu_1<0$ and $n\ge 2$.  Suppose $h\in C^2((0,\infty))\cap C^1([0,\infty))$ is  a solution  of \eqref{h-steady-soliton-eqn}.  Then
\begin{equation}\label{r-h-infty-limit}
b_1:=\lim_{r\to\infty}rh(r)\in [0,\infty)\quad\mbox{ exists}
\end{equation}
with $b_1>0$ when $n\ge 4$. Moreover
\begin{equation}\label{h-2nd-order-limit} 
\lim_{r\to\infty}r\left(\frac{rh_r(r)}{h(r)}+1\right)=\left(\frac{n-4}{n-1}\right)b_1.
\end{equation}
\end{thm}

\begin{thm}\label{expanding-soln-asymptotic-hr-behaviour-thm}
Let $\lambda>0$, $\mu_1\in\mathbb{R}\setminus\{0\}$ and $n\ge 2$.  Suppose $h\in C^2((0,\infty))\cap C^1([0,\infty))$ is  a solution  of \eqref{h-ode-initial-value-problem}. Then the following holds.

\begin{enumerate}[(i)] 

\item If $\mu_1=\lambda/n$, then 
\begin{equation}\label{h-uniform-upper-lower-bd31}
h(r)=1+\frac{\lambda}{n}r\quad\forall r\ge 0.
\end{equation}
is the explicit analytic solution of \eqref{h-ode-initial-value-problem}.

\item If $0<\mu_1<\lambda/n$, then  $\underset{\substack{{r\to\infty}}}{\lim}\,h_r(r)=0$ and $\underset{\substack{{r\to\infty}}}{\lim}\,h(r)\in (1,\infty)$ exists and
\begin{equation}\label{h-uniform-upper-lower-bd32}
1<h(r)\le 1+\mu_1r\quad\forall r>0.
\end{equation}

\item If $\mu_1>\lambda/n$,  then 
\begin{equation}\label{h-hr-limit-at-infty}
\lim_{r\to\infty}h_r(r)=\lim_{r\to\infty}h(r)=\infty
\end{equation}
 and
\begin{equation}\label{h-uniform-upper-lower-bd33}
h(r)\ge 1+\mu_1r\quad\forall r>0.
\end{equation}

\item If $\mu_1<0$, then $\underset{\substack{{r\to\infty}}}{\lim}\,h_r(r)=0$ and $\underset{\substack{{r\to\infty}}}{\lim}\,h(r)\in (0,1)$.
\end{enumerate}

\end{thm}

\begin{thm}\label{steady-soln-asymptotic-behaviour-for-hr-positive-thm}
Let $n\ge 2$ and $\mu_1>0$.  Suppose $h\in C^2((0,\infty))\cap C^1([0,\infty))$ is  a solution  of \eqref{h-steady-soliton-eqn}. Then $h$ satisfies \eqref{h-hr-limit-at-infty} and \eqref{h-uniform-upper-lower-bd33}.
\end{thm}

The plan of the paper is as follows. In section two we will use fixed point technique to prove the existence of unique solution of \eqref{h-ode-initial-value-problem}. In section three we will use a modification of the power series method of \cite{Br} to prove Theorem \ref{h-existence-analytic-soln-thm}. In section four we will use integral representation of  $h$ and construction of  appropriate auxilliary functions to prove the asymptotic behaviour of the solution of \eqref{h-steady-soliton-eqn}. In section five we will prove Theorem \ref{expanding-soln-asymptotic-hr-behaviour-thm} and Theorem \ref{steady-soln-asymptotic-behaviour-for-hr-positive-thm}. 

\section{Existence of rotationally symmetric gradient Ricci soliton}
\setcounter{equation}{0}
\setcounter{thm}{0}

In this section we will use fixed point argument to prove  the existence of solution of \eqref{h-ode-initial-value-problem}. We first start with the local existence of solution of \eqref{h-eqn}, \eqref{h-initial-condition}, near the origin. 

\begin{lem}\label{h-local-existence-lem}
Let $\lambda, \mu_1\in\mathbb{R}$ and $2\le n\le 4$. Then there exists a constant $\3>0$ such that 
\eqref{h-eqn} has a unique solution $h\in C^2((0,\3))\cap C^1([0,\3))$ in $(0,\3)$ which satisfies
\eqref{h-initial-condition}.
\end{lem} 
\begin{proof}
We first observe that if $h\in C^2((0,\3))\cap C^1([0,\3))$ is a solution of \eqref{h-eqn} in $(0,\3)$ for some constant $\3>0$ which satisfies \eqref{h-initial-condition}, then 
\begin{align}
&2r^2h(r)h_{rr}(r)=(n-1)(h(r)-1)^2+r^2\left\{h_r(r)^2-\lambda h_r(r)+\frac{n-1}{r}\left(\frac{h(r)-1}{r}-h_r(r)\right)\right\}\notag\\
\Leftrightarrow\quad&h_{rr}(r)=\frac{1}{2h(r)}\left\{\frac{n-1}{r^2}\left(\int_0^rh_r(\rho)\,d\rho\right)^2+h_r(r)^2-\lambda h_r(r)+\frac{n-1}{r^2}\int_0^r(h_r(\rho)-h_r(r))\,d\rho\right\}\label{h-equivalent-eqn}\\
\Leftrightarrow\quad&h_r(r)=\int_0^r\frac{1}{2h(s)}\left\{\frac{n-1}{s^2}\left(\int_0^sh_r(\rho)\,d\rho\right)^2+\frac{n-1}{s^2}\int_0^s(h_r(\rho)-h_r(s))\,d\rho +h_r(s)^2-\lambda h_r(s)\right\}\,ds\notag\\
&\qquad\quad +\mu_1\notag
\end{align}
for any $0<r<\3$. This suggests one to use fixed point argument to prove the local existence of solution of \eqref{h-ode-initial-value-problem}.

For any $\3>0$ we now define the Banach space 
$$
\cX_\3:=\left\{(h,w): h\in  C^{0,1}\left( [0,\3]; \mathbb{R}\right),w\in C^{0,1}\left( [0,\3]; \mathbb{R}\right)\right\}
$$ 
with a norm given by
$$||(h,w)||_{\cX_\3}=\max\left(\|h\|_{C^{0,1}\left([0, \3]\right)} ,\|w\|_{C^{0,1}\left([0, \3]\right)} \right)$$  
where
\begin{equation*}
\|w\|_{C^{0,1}\left([0, \3]\right)}=\max(\|w\|_{L^{\infty}\left([0, \3]\right)},\|w\|_{1,[0, \3]})
\quad\mbox{ and }\quad  \|w\|_{1,[0, \3]}=\sup_{s,s'\in[0,\ve],\, s\not=s'} \frac{|w(s)-w(s')|}{|s-s'|}.
\end{equation*}
For any $(h,w)\in \cX_\3,$ we define  
$$\Phi(h,w):=\left(\Phi_1(h,w),\Phi_2(h,w)\right),$$ 
where 
\begin{equation}\label{contraction-map-defn}
\left\{\begin{aligned}
\Phi_1(h,w)(r)=&1+\int_0^r w(\rho)\,d\rho,\\
\Phi_2(h,w)(r)=&\int_0^r\frac{1}{2h(s)}\left\{(n-1)\left(\frac{\left(\int_0^sw(\rho)\,d\rho\right)^2}{s^2}+\frac{\int_0^s(w(\rho)-w(s))\,d\rho}{s^2}\right)+w(s)^2-\lambda w(s)\right\}\,ds\\
&\quad +\mu_1\end{aligned}\right.
\end{equation}
for any $0<r\leq\3$. Let $\cD_{\3}$ be the family of all $(h,w)\in \cX_\3$ satisfying
\begin{equation}\label{closed-set-defn}
\left\{\begin{aligned}
&\|h-1\|_{L^\infty([0, \3])}\leq (|\mu_1|+1)\3\\
&\|h\|_{1,[0, \3]}\le |\mu_1|+1\\
&\|w-\mu_1\|_{L^\infty([0, \3])}\leq 3\left(n(|\mu_1|+1)^2+|\lambda|(|\mu_1|+1)\right)\3\\
&\|w\|_{1,[0, \3]}\le 3\left(n(|\mu_1|+1)^2+|\lambda|(|\mu_1|+1)\right)\\
&h(0)=1,\quad w(0)=\mu_1.
\end{aligned}\right.
\end{equation}
Note that $\cD_{\3}$  is a closed subspace of $\cX_\3$. Since $(1,\mu_1)\in \cD_{\3}$, $\cD_{\3}\ne\phi$. Let
\begin{equation*}
\3_1=\frac{1}{100\left(n(|\mu_1|+1)^2+|\lambda|(|\mu_1|+1)\right)}
\end{equation*}
We will assume that $\3\in(0,\3_1)$ for the rest of the proof.
We will show that if $\3\in(0,\3_1)$ is sufficiently small, the map $(h,w)\mapsto\Phi(h,w)$ will have a unique fixed point in $\cD_{\3}.$
We first  prove that $\Phi(\cD_{\3})\subset \cD_{\3}$ for any $\3\in(0,\3_1)$. In fact for any $\3\in(0,\3_1)$ and $(h,w)\in \cD_{\3},$ by \eqref{closed-set-defn} we have
\begin{equation}\label{h-w-upper-lower-bds}
\frac{99}{100}\le h(r)\le\frac{101}{100}\quad\forall 0\le r\le\3\quad\mbox{ and }\quad  \|w\|_{L^{\infty}\left([0, \3]\right)}\le |\mu_1|+\frac{3}{100}\le |\mu_1|+1.
\end{equation}
Hence by \eqref{contraction-map-defn}, \eqref{closed-set-defn} and \eqref{h-w-upper-lower-bds}, we have
\begin{align}\label{phi-1-lip}
&|\Phi_1(h,w)(r_2)-\Phi_1(h,w)(r_1)|\le \|w\|_{L^{\infty}([0, \3])}|r_2-r_1|\quad\forall 0\le r_1<r_2\le\3\notag\\
\Rightarrow\quad&\|\Phi_1(h,w)\|_{1,[0, \3]}\le \|w\|_{L^{\infty}([0, \3])}\le |\mu_1|+1
\end{align}
and
\begin{equation}\label{c1-condition}
\|\Phi_1(h,w)-1\|_{L^{\infty}\left([0, \3]\right)}\le\|w\|_{L^{\infty}\left([0, \3)\right)}\3
\le(|\mu_1|+1)\3
\end{equation}
and
\begin{align*}
&|\Phi_2(h,w)(r_2)-\Phi_2(h,w)(r_1)|\notag\\
\le&\frac{100}{2\cdot 99}\int_{r_1}^{r_2}
\left(n\|w\|^2_{L^{\infty}([0, \3])}+|\lambda|\|w\|_{L^{\infty}([0, \3])}+\frac{(n-1)\|w\|_{1,[0, \3]}}{s^2}\int_0^s(s-\rho)\,d\rho\right)\,ds\notag\\
\le&\frac{50}{99}
\left(n(|\mu_1|+1)^2+|\lambda|(|\mu_1|+1)+\frac{3(n-1)}{2}\left(n(|\mu_1|+1)^2+|\lambda|(|\mu_1|+1)\right)\right)|r_2-r_1|\notag\\
\le&\frac{25}{9}\left(n(|\mu_1|+1)^2+|\lambda|(|\mu_1|+1)\right)|r_2-r_1|\notag\\
\le&3\left(n(|\mu_1|+1)^2+|\lambda|(|\mu_1|+1)\right)|r_2-r_1|\quad\forall 0\le r_1<r_2\le\3
\end{align*}
Thus
\begin{equation}\label{c3-condition}
\|\Phi_2(h,w)\|_{1,[0, \3]}\le 3\left(n(|\mu_1|+1)^2+|\lambda|(|\mu_1|+1)\right)
\end{equation}
and
\begin{equation}\label{c2-condition}
|\Phi_2(h,w)(r)-\mu_1|\le 3\left(n(|\mu_1|+1)^2+|\lambda|(|\mu_1|+1)\right)\3\qquad\forall 0\le r\le\3.
\end{equation}
By \eqref{phi-1-lip}, \eqref{c1-condition}, \eqref{c3-condition} and \eqref{c2-condition} we get that $\Phi(\cD_{\3})\subset \cD_{\3}$ for any $\3\in(0,\3_1)$.

Let $(h_1,w_1), (h_2,w_2)\in \cD_{\3}$, $0<\3<\3_1$, $\delta_1=\|(h_1,w_1)-(h_2,w_2)\|_{\cX_\3}$ and
\begin{equation}
E(w,s)=\frac{n-1}{s^2}\left(\int_0^sw(\rho)\,d\rho\right)^2+\frac{n-1}{s^2}\int_0^s(w(\rho)-w(s))\,d\rho+w(s)^2-\lambda w(s)
\end{equation}
for any  $w\in C^{0,1}\left( [0,\3]; \mathbb{R}\right)$, $0<s<\3$. Then by \eqref{closed-set-defn},
\begin{equation}\label{h12-w-upper-lower-bds}
\frac{99}{100}\le h_i(r)\le \frac{101}{100}\quad\quad\mbox{ and }\quad  \|w_i\|_{L^{\infty}\left([0, \3]\right)}\le |\mu_1|+1\quad\forall 0\le r\le\3,i=1,2.
\end{equation}
and since $w_1(0)=w_2(0)=\mu_1$,
\begin{align}\label{e-w12-compare}
&|E(w_1,s)-E(w_2,s)|\notag\\
\le&\frac{n-1}{s^2}\left(\int_0^s|w_1(\rho)-w_2(\rho)|\,d\rho\right)
\left(\int_0^s|w_1(\rho)+w_2(\rho)|\,d\rho\right)+|w_1(s)-w_2(s)||w_1(s)+w_2(s)|\notag\\
&\qquad +\frac{n-1}{s^2}\int_0^s|(w_1-w_2)(\rho)-(w_1-w_2)(s)|\,d\rho+|\lambda||w_1(s)-w_2(s)|\notag\\
\le&\left\{\left(\frac{n-1}{s}\int_0^s\rho\,d\rho+s\right)
(\|w_1\|_{L^{\infty}([0,\3])}+\|w_2\|_{L^{\infty}([0,\3])})
+\frac{n-1}{s^2}\int_0^s(s-\rho)\,d\rho+|\lambda|s\right\}\|w_1-w_2\|_{1,[0,\3]}\notag\\
\le&\left(((n+1)(|\mu_1|+1)+|\lambda|)s+\frac{n-1}{2}\right)\|w_1-w_2\|_{1,[0,\3]}\quad\forall 0<s\le\3.
\end{align}
Similarly,
\begin{align}\label{e-expression-bd}
|E(w_2,s)|\le&n\|w_2\|_{L^{\infty}([0,\3])}^2+|\lambda|\|w_2\|_{L^{\infty}([0,\3])}+\frac{n-1}{2}\|w_2\|_{1,[0,\3]}\notag\\
\le&n(|\mu_1|+1)^2+|\lambda|(|\mu_1|+1)+\frac{3(n-1)}{2}\left(n(|\mu_1|+1)^2+|\lambda|(|\mu_1|+1)\right)\notag\\
\le&\frac{11}{2}\left(n(|\mu_1|+1)^2+|\lambda|(|\mu_1|+1)\right)=:C_1\quad (\mbox{say})\quad\forall 0<s\le\3.
\end{align}
Hence by \eqref{contraction-map-defn}, \eqref{closed-set-defn}, \eqref{h12-w-upper-lower-bds}, \eqref{e-w12-compare} and \eqref{e-expression-bd},
\begin{align*}
&|(\Phi_2(h_1,w_1)-\Phi_2(h_2,w_2))(r_2)-(\Phi_2(h_1,w_1)-\Phi_2(h_2,w_2))(r_1)|\notag\\
=&\frac{1}{2}\left|\int_{r_1}^{r_2}\frac{E(w_1,s)}{h_1(s)}\,ds-\int_{r_1}^{r_2}\frac{E(w_2,s)}{h_2(s)}\,ds\right|\notag\\
\le&\int_{r_1}^{r_2}\frac{|E(w_1,s)-E(w_2,s)|}{2h_1(s)}\,ds+\int_{r_1}^{r_2}\frac{|E(w_2,s)||h_2(s)-h_1(s)|}{2h_1(s)h_2(s)}\,ds\notag\\
\le&\frac{100}{2\cdot 99}\left(\int_{r_1}^{r_2}\left(((n+1)(|\mu_1|+1)+|\lambda|)s+\frac{n-1}{2}\right)\,ds\right)\|w_1-w_2\|_{1,[0,\3]}\notag\\
&\qquad+\frac{1}{2}\left(\frac{100}{99}\right)^2C_1\left(\int_{r_1}^{r_2}s\,ds\right)\|h_1-h_2\|_{1,[0,\3]}\notag\\
\le&\left\{\frac{25}{99}\left[((n+1)(|\mu_1|+1)+|\lambda|)(r_2+r_1)+n-1\right]+\frac{2500C_1}{99^2}(r_2+r_1)\right\}|r_2-r_1|\delta_1\notag\\
\le&\left(C_2\3+\frac{25}{33}\right)|r_2-r_1|\delta_1\quad\forall 0\le r_1<r_2\le\3
\end{align*}
where
\begin{equation*}
C_2=(n+1)(|\mu_1|+1)+|\lambda|+\frac{5000}{99^2}C_1.
\end{equation*}
Thus
\begin{equation}\label{phi2-1-estimate}
\|(\Phi_2(h_1,w_1)-\Phi_2(h_2,w_2))\|_{1,[0,\3]}\le\left(C_2\3+\frac{25}{33}\right)\delta_1.
\end{equation}
Similarly,
\begin{equation}\label{phi2-0-estimate}
\|(\Phi_2(h_1,w_1)-\Phi_2(h_2,w_2))\|_{L^{\infty}([0,\3])}\le\left(C_2\3+\frac{25}{33}\right)\3\delta_1.
\end{equation}
Let $\3_2=\min \left(\3_1,\frac{1}{33C_2}\right)$ and $0<\3<\3_2$ for the rest of the proof. Then by \eqref{phi2-1-estimate} and \eqref{phi2-0-estimate},
\begin{equation}\label{phi2-c01-estimate}
\|(\Phi_2(h_1,w_1)-\Phi_2(h_2,w_2))\|_{C^{0,1}([0,\3])}\le\left(C_2\3+\frac{25}{33}\right)\delta_1\le\frac{26}{33}\delta_1.
\end{equation}
On the other hand,
\begin{align*}
&|(\Phi_1(h_1,w_1)-\Phi_1(h_2,w_2))(r_2)-(\Phi_2(h_1,w_1)-\Phi_2(h_2,w_2))(r_1)|\\
\le&\int_{r_1}^{r_2}|w_1(s)-w_2(s)|\,ds\notag\\
\le&\|w_1-w_2\|_{1,[0,\3]}\int_{r_1}^{r_2}s\,ds\notag\\
\le&\3\delta_1|r_2-r_1|\quad\forall 0\le r_1<r_2\le\3.
\end{align*}
Hence
\begin{equation}\label{phi1-lip-estimate}
\|\Phi_1(h_1,w_1)-\Phi_1(h_2,w_2)\|_{1,[0,\3]}\le\3\delta_1.
\end{equation}
Similarly,
\begin{align}\label{phi1-c0-estimate}
&|\Phi_1(h_1,w_1)(r)-\Phi_1(h_2,w_2)(r)|\le\3\delta_1 r\le\3^2\delta_1\quad\forall0\le r\le\3\notag\\
\Rightarrow\quad&\|\Phi_1(h_1,w_1)-\Phi_1(h_2,w_2)\|_{L^{\infty}([0,\3])}\le\3\delta_1.
\end{align}
By \eqref{phi1-lip-estimate} and \eqref{phi1-c0-estimate},
\begin{equation}\label{phi1-c01-estimate}
\|\Phi_1(h_1,w_1)-\Phi_1(h_2,w_2)\|_{C^{0,1}([0,\3])}\le\3\delta_1.
\end{equation}
Hence for any $0<\3<\3_2$, by \eqref{phi2-c01-estimate} and \eqref{phi1-c01-estimate},
\begin{equation*}
\|(\Phi(h_1,w_1)-\Phi(h_2,w_2))\|_{\cX_\3}\le\frac{26}{33}\delta_1.
\end{equation*}
Thus $\Phi$ is a contraction map on  $\cD_{\3}$. Hence by the contraction map theorem there exists a 
unique fixed point $(h,w)=\Phi(h,w)$ in $\cD_{\3}$. Thus
\begin{equation}\label{fix-pt-eqn}
\left\{\begin{aligned}
h(r)=&1+\int_0^r w(\rho)\,d\rho,\\
w(r)=&\int_0^r\frac{1}{2h(s)}\left\{\frac{n-1}{s^2}\left(\int_0^sw(\rho)\,d\rho\right)^2+\frac{n-1}{s^2}\int_0^s(w(\rho)-w(s))\,d\rho+w(s)^2-\lambda w(s)\right\}\,ds\\
&\quad+\mu_1
\end{aligned}\right.
\end{equation}
for any $0<r\leq\3$ with $h(0)=1$, $w(0)=\mu_1$, and both $h$ and $w$ are differentiable on $[0,\3]$ with
\begin{equation}\label{hr-w-relation}
h_r(r)=w(r)
\end{equation}
and
\begin{equation}\label{w-integral-eqn}
w_r(r)=\frac{1}{2h(r)}\left\{\frac{n-1}{r^2}\left(\int_0^rw(\rho)\,d\rho\right)^2+\frac{n-1}{r^2}\int_0^r(w(\rho)-w(r))\,d\rho+w(r)^2-\lambda w(s)\right\}.
\end{equation}
Putting \eqref{hr-w-relation} in \eqref{w-integral-eqn} we get \eqref{h-equivalent-eqn}. Hence for any $0<\3<\3_2$,  \eqref{h-eqn} has a unique solution $h\in C^2((0,\3))\cap C^1([0,\3))$  in $(0,\3)$ which satisfies \eqref{h-initial-condition} and the lemma follows.
\end{proof}

We next observe that by a similar argument as the proof of Lemma 1 of \cite{Br} we have the following result.

\begin{lem}(cf. Lemma 1 of \cite{Br})\label{h-monotone-lem}
Let $\lambda\in\mathbb{R}$, $0\le a<b$ and $n\ge 2$. Let $h\in C^2((a,b))$ be a solution of \eqref{h-eqn} in $(a,b)$. Then the following holds.

\begin{enumerate}

\item[(i)] Suppose there exists $x_0\in (a,b)$ such that $h_r(x_0)\ge 0$ and $h(x_0)>1$. Then $h_r(x)>0$ on $(x_0,b)$. 

\item[(ii)] Suppose there exists $x_0\in (a,b)$ such that $h_r(x_0)\le 0$ and $h(x_0)<1$. Then $h_r(x)<0$ on $(x_0,b)$. 

\end{enumerate}
\end{lem}

\begin{lem}\label{h-lower-bd-lem}
Let $\lambda\in\mathbb{R}$ and $n\ge 2$. Suppose $h\in C^2((0,L))$ satisfies \eqref{h-eqn}
in $(0,L)$ for some constant $L\in (0,\infty)$  such that $L<-\frac{(n-1)}{\lambda}$ if $\lambda<0$. Then there exists a constant $c_1>0$ such that
\begin{equation}\label{h-uniform-lower-bd}
h(r)\ge c_1\quad\forall L/2\le r\le L.
\end{equation}  
\end{lem}
\begin{proof}
Suppose \eqref{h-uniform-lower-bd} does not hold. Then there exists a sequence $\{r_i\}_{i=1}^{\infty}\subset (L/2,L)$, $r_i<r_{i+1}$ for all $i\in\Z^+$,  $r_i\to L$ as $i\to\infty$, such that
\begin{equation}\label{h-infty-limit=0}
h(r_i)\to 0\quad\mbox{ as }i\to\infty.
\end{equation}
Without loss of generality we may assume that
\begin{equation}\label{h-monotone-decreasing}
0<h(r_{i+1})<h(r_i)<1\quad\forall i\in\Z^+.
\end{equation}
By \eqref{h-eqn} or Lemma 3.1 of \cite{H},
\begin{align*}
(h^{-1/2}h_r)_r=&(n-1)\frac{h-1}{2r^2h^{1/2}}-\frac{n-1+\lambda r}{2r}\cdot\frac{h_r}{h^{3/2}}
\quad\forall 0<r<L\notag\\
\Rightarrow\qquad h_r(r)=&\sqrt{h(r)}\left\{\frac{h_r(r_1)}{\sqrt{h(r_1)}}+(n-1)\int_{r_1}^r\frac{h(\rho)-1}{2\rho^2\sqrt{h(\rho)}}\,d\rho-\int_{r_1}^r\frac{((n-1)\rho^{-1}+\lambda)h_r(\rho)}{2h(\rho)^{3/2}}\,d\rho\right\}\notag\\
=&\sqrt{h(r)}\left\{\frac{h_r(r_1)}{\sqrt{h(r_1)}}+(n-1)\int_{r_1}^r\frac{h(\rho)-1}{2\rho^2\sqrt{h(\rho)}}\,d\rho+\left(\frac{n-1}{r}+\lambda\right)\frac{1}{\sqrt{h(r)}}\right.\notag\\
&\qquad\left.-\left(\frac{n-1}{r_1}+\lambda\right)\frac{1}{\sqrt{h(r_1)}}+(n-1)\int_{r_1}^r\frac{d\rho}{\rho^2\sqrt{h(\rho)}}\right\}\quad\forall r_1<r<L.\notag
\end{align*}
Hence
\begin{align}
h_r(r)=&\frac{n-1}{r}+\lambda+\sqrt{\frac{h(r)}{h(r_1)}}\left(h_r(r_1)-\frac{n-1}{r_1}-\lambda\right)+(n-1)\frac{\sqrt{h(r)}}{2}\int_{r_1}^r\frac{h(\rho)+1}{\rho^2\sqrt{h(\rho)}}\,d\rho\label{h-derivative-integral-formula0}\\
>&\frac{n-1}{r}+\lambda+\sqrt{\frac{h(r)}{h(r_1)}}\left(h_r(r_1)-\frac{n-1}{r_1}-\lambda\right)\quad\forall r_1<r<L.\label{h-derivative-integral-formula}
\end{align}
Putting $r=r_i$ in \eqref{h-derivative-integral-formula} and letting $i\to\infty$,
\begin{align}\label{h-derivative-ri-positive}
&\liminf_{i\to\infty}h_r(r_i)\ge\frac{n-1}{L}+\lambda>0\notag\\
\Rightarrow\quad&\exists i_0\in\Z^+\,\,\mbox{ such that }\,\, h_r(r_i)>0\quad\forall i\ge i_0.
\end{align}
By \eqref{h-monotone-decreasing} and the mean value thoerem for any $i\in\Z^+$ there exists $\xi_i\in (r_i,r_{I+1})$ such that 
\begin{equation*}
h_r(\xi_i)=\frac{h(r_{i+1})-h(r_i)}{r_{i+1}-r_i}<0.
\end{equation*}
This together with \eqref{h-derivative-ri-positive} implies that for any $i\ge i_0$, there exists a maximal interval $(r_i,b_i)$, $b_i\in (r_i,\xi_i)$, such that 
\begin{equation}\label{hr-0}
h_r(r)>0\quad\forall r_i\le r<b_i\quad\mbox{ and }\quad h_r(b_i)=0.
\end{equation}
We now divide the proof into three cases.

\noindent $\underline{\text{\bf Case 1}}$: $h(b_{i_0})=1$.

\noindent By \eqref{hr-0} and uniqueness solution of ODE $h(r)\equiv 1$ on $(0,L)$. This contradicts \eqref{h-infty-limit=0}. Hence case 1 does not hold.

\noindent $\underline{\text{\bf Case 2}}$: $h(b_{i_0})>1$.

\noindent Then by \eqref{hr-0} and Lemma \ref{h-monotone-lem},
\begin{equation*}
h_r(r)>0\quad\forall b_{i_0}<r<L\quad
\Rightarrow\quad h(r)>h(b_{i_0})\quad \forall b_{i_0}<r<L
\end{equation*}
which contradicts \eqref{h-infty-limit=0}. Hence case 2 does not hold.

\noindent $\underline{\text{\bf Case 3}}$: $h(b_{i_0})<1$.

\noindent Then by \eqref{hr-0} and Lemma \ref{h-monotone-lem},
\begin{equation*}
h_r(r)<0\quad\forall b_{i_0}\le r<L.
\end{equation*}
This contradicts \eqref{h-derivative-ri-positive}. Hence case 3 does not hold. By case 1, case 2 and case 3
we get a contradiction. Hence no such sequence $\{r_i\}_{i=1}^{\infty}$ exists. Thus there exists a constant $c_1>0$ such that \eqref{h-uniform-lower-bd} holds and the lemma follows.
\end{proof}

\begin{lem}\label{h-upper-bd-lem}
Let $\lambda\in\mathbb{R}$ and $n\ge 2$. Suppose $h\in C^2((0,L))$ satisfies \eqref{h-eqn}
in $(0,L)$ for some constant $L\in (0,\infty)$  such that $L<-(n-1)/\lambda$ if $\lambda<0$. Then there exists a constant $c_2>0$ such that
\begin{equation}\label{h-uniform-upper-bd}
h(r)\le c_2\quad\forall L/2\le r\le L.
\end{equation} 
\end{lem}
\begin{proof}
Suppose \eqref{h-uniform-upper-bd} does not hold. Then there exists a sequence $\{r_i\}_{i=1}^{\infty}\subset (L/2,L)$, $r_i<r_{i+1}$ for all $i\in\Z^+$, $r_i\to L$ as $i\to\infty$, such that
\begin{equation}\label{h-infty-limit=infty}
h(r_i)\to\infty\quad\mbox{ as }i\to\infty.
\end{equation}
Without loss of generality we may assume that
\begin{equation}\label{h-monotone-increasing}
h(r_{i+1})>h(r_i)+1\quad\forall i\in\Z^+.
\end{equation}
By \eqref{h-monotone-increasing} and the mean value theorem for any $i\in\Z^+$ there exists $\xi_i\in (r_i,r_{i+1})$ such that
\begin{align}\label{h-derivatire-positive+limit-infty}
&h_r(\xi_i)=\frac{h(r_{i+1})-h(r_{r_i})}{r_{i+1}-r_i}>\frac{1}{r_{i+1}-r_i}>0\quad\forall i\in\Z^+\notag\\
\Rightarrow\quad&h_r(\xi_i)\to\infty\quad\mbox{ as }i\to\infty.
\end{align}
Hence by \eqref{h-derivatire-positive+limit-infty} there exists $i_0\in\Z^+$ such that
\begin{equation}\label{h-derivative-lower-bd1}
h_r(\xi_{i_0})>\frac{2(n-1)}{L}+\lambda\quad\Rightarrow\quad h_r(\xi_{i_0})>\frac{n-1}{\xi_{i_0}}+\lambda.
\end{equation}
By replacing $r_1$ by $\xi_{i_0}$ in \eqref{h-derivative-integral-formula} we get
\begin{equation}\label{h-derivative-lower-bd2}
h_r(r)>\frac{n-1}{r}+\lambda+\sqrt{\frac{h(r)}{h(\xi_{i_0})}}\left(h_r(\xi_{i_0})-\frac{n-1}{\xi_{i_0}}-\lambda\right)>\frac{n-1}{r}+\lambda>0\quad\forall \xi_{i_0}\le r<L.
\end{equation}
By \eqref{h-infty-limit=infty} and \eqref{h-derivative-lower-bd2},
\begin{equation}\label{h-endpt-limit=infty}
h(r)\to\infty\quad\mbox{ as }r\to L.
\end{equation}
Since $L<-(n-1)/\lambda$ if $\lambda<0$,
\begin{equation}\label{lambda r+}
n-1+\lambda r>0\quad\forall 0<r<L.
\end{equation}
By \eqref{h-eqn}, \eqref{h-derivative-lower-bd2} and \eqref{lambda r+},
\begin{equation}\label{hr-h-ratio-diff-lower-bd}
\left(\frac{h_r}{h}\right)_r=\frac{h_{rr}}{h}-\frac{h_r^2}{h^2}=\frac{1}{2h}\left((n-1)\frac{h-1}{r^2}-\frac{(n-1+\lambda r)h_r}{rh}+\frac{h_r^2}{h}\right)-\frac{h_r^2}{h^2}
\le(n-1)\frac{h-1}{2r^2h}\le\frac{2(n-1)}{L^2}
\end{equation}
holds for any  $\xi_{i_0}\le r<L$.
Integrating \eqref{hr-h-ratio-diff-lower-bd} over $(\xi_{i_0},r)$, $r\in (\xi_{i_0},L)$, by \eqref{h-derivative-lower-bd2} we have
\begin{align*}
&\frac{h_r(r)}{h(r)}\le\frac{h_r(\xi_{i_0})}{h(\xi_{i_0})}+\frac{2(n-1)}{L}=:C_3\,\,(\mbox{say})\quad\forall \xi_{i_0}\le r<L\notag\\
\Rightarrow\quad&\log\left(\frac{h(r)}{h(\xi_{i_0})}\right)\le C_3L\quad\forall \xi_{i_0}\le r<L\notag\\
\Rightarrow\quad&h(r)\le h(\xi_{i_0})e^{C_3L}\quad\forall \xi_{i_0}\le r<L.
\end{align*}
which contradicts \eqref{h-endpt-limit=infty}. Hence no such sequence $\{r_i\}_{i=1}^{\infty}$ exists. Thus there exists a constant $c_2>0$ such that \eqref{h-uniform-upper-bd} holds and the lemma follows.

\end{proof}

\begin{lem}\label{h-derivative-upper-bd-lem}
Let $\lambda\in\mathbb{R}$ and $n\ge 2$. Suppose $h\in C^2((0,L))$ satisfies \eqref{h-eqn}
in $(0,L)$ for some constant $L\in (0,\infty)$  such that $L<-(n-1)/\lambda$ if $\lambda<0$. Then there exists a constant $c_3>0$ such that
\begin{equation}\label{h-derivative-uniform-upper-bd}
h_r(r)\le c_3\quad\forall L/2\le r\le L.
\end{equation}  
\end{lem}
\begin{proof}
Let $c_2>0$ be as given by Lemma \ref{h-upper-bd-lem}.
Suppose \eqref{h-derivative-uniform-upper-bd}  does not hold. Then there exists a sequence $\{r_i\}_{i=1}^{\infty}\subset (L/2,L)$, $r_i<r_{i+1}$ for all $i\in\Z^+$, $r_i\to L$ as $i\to\infty$, such that
\begin{equation}\label{h-derivative-endpt-limit=infty}
h_r(r_i)\to\infty\quad\mbox{ as }i\to\infty.
\end{equation}
Let
\begin{equation}\label{c4-defn}
C_4=\frac{16}{L}\max\left(c_2n,\sqrt{c_2n},n-1+|\lambda| L\right).
\end{equation}
By  \eqref{h-derivative-endpt-limit=infty} there exists $i_0\in\Z^+$ such that
\begin{equation}\label{hr-lower-bd3}
h_r(r_i)>C_4\quad\forall i\ge i_0.
\end{equation}
We now rewrite \eqref{h-eqn} as
\begin{equation}\label{h-eqn2}
h_{rr}(r)=\frac{h_r(r)^2}{2h(r)}+(n-1)\frac{h(r)-1}{2r^2}-\frac{n-1+\lambda r}{2r}\cdot\frac{h_r(r)}{h(r)},\quad h(r)>0.
\end{equation}
Then by  \eqref{h-uniform-upper-bd},  \eqref{hr-lower-bd3} and \eqref{h-eqn2},
\begin{align}\label{hrr-lower-bd-ineqn}
h_{rr}(r_{i_0})=&\frac{h_r(r_{i_0})^2}{4h(r_{i_0})}+\frac{1}{2}\left(\frac{h_r(r_{i_0})^2}{4h(r_{i_0})}+(n-1)\frac{h(r_{i_0})-1}{r_{i_0}^2}\right)
+\frac{1}{2h(r_{i_0})}\left(\frac{h_r(r_{i_0})^2}{4}-\frac{n-1+\lambda r_{i_0}}{r_{i_0}}h_r(r_{i_0})\right)\notag\\
\ge&\frac{h_r(r_{i_0})^2}{4h(r_{i_0})}+\frac{1}{2}\left(\frac{h_r(r_{i_0})^2}{4c_2}-\frac{4(n-1)}{L^2}\right)
+\frac{1}{2h(r_{i_0})}\left(\frac{h_r(r_{i_0})^2}{4}-\frac{h_r(r_{i_0})^2}{16}-16\left(\frac{(n-1+|\lambda| L)}{L}\right)^2\right)\notag\\
>&\frac{h_r(r_{i_0})^2}{4h(r_{i_0})}.
\end{align}
Hence by \eqref{hrr-lower-bd-ineqn} and continuity there exists $0<\delta<L-r_{i_0}$ such that
\begin{equation}\label{hrr-hr-hr-ineqn}
h_{rr}(r)>\frac{h_r(r)^2}{4h(r)}\quad\forall r_{i_0}\le r<r_{i_0}+\delta.
\end{equation}
Let 
\begin{equation*}
b_0=\sup\left\{b\in (r_{i_0},L):h_{rr}(r)>\frac{h_r(r)^2}{4h(r)}\quad\forall r_{i_0}\le r<b\right\}.
\end{equation*}
Then $b_0\ge r_{i_0}+\delta$. Suppose $b_0<L$. Then by \eqref{hr-lower-bd3} and the definition of $b_0$,
\begin{align}
&h_{rr}(r)>\frac{h_r(r)^2}{4h(r)}\quad\forall r_{i_0}\le r<b_0\quad\mbox{ and }\quad h_{rr}(b_0)=\frac{h_r(b_0)^2}{4h(b_0)}\label{hrr-hr-relation=}\\
\Rightarrow\quad&h_r(b_0)>h_r(r_{i_0})>C_4.\label{hrr-hr-relation>}
\end{align}
Then by repeating the above argument but with $b_0$ and \eqref{hrr-hr-relation>} replacing $r_{i_0}$ and 
\eqref{hr-lower-bd3} in the proof we get
\begin{equation*}
h_{rr}(b_0)>\frac{h_r(b_0)^2}{4h(b_0)}
\end{equation*}
which contradicts \eqref{hrr-hr-relation=}. Hence $b_0=L$. By \eqref{hr-lower-bd3},
\begin{align}\label{hrr-hr-relation>5}
&h_{rr}(r)>\frac{h_r(r)^2}{4h(r)}\quad\forall r_{i_0}\le r<L\notag\\
\Rightarrow\quad&h_r(r)>h_r(r_{i_0})>C_4\quad\forall r_{i_0}\le r<L.
\end{align}
By \eqref{h-uniform-upper-bd}, \eqref{lambda r+}, \eqref{h-eqn2} and \eqref{hrr-hr-relation>5}, 
\begin{align*}
&h_{rr}(r)\le\frac{1}{2h(r)}\left(h_r(r)^2+(n-1)\frac{h(r)^2}{r^2}\right)\le\frac{1}{2h(r)}\left(h_r(r)^2+\frac{4(n-1)c_2^2}{L^2}\right)\le\frac{h_r(r)^2}{h(r)}\quad \forall r_{i_0}\le r<L\notag\\
\Rightarrow\quad&\frac{h_{rr}(r)}{h_r(r)}\le\frac{h_r(r)}{h(r)}\quad \forall r_{i_0}\le r<L\notag\\
\Rightarrow\quad&h_r(r)\le\frac{h_r(r_{i_0})}{h(r_{i_0})}h(r)\le c_2\frac{h_r(r_{i_0})}{h(r_{i_0})}\quad \forall r_{i_0}\le r<L
\end{align*}
which contradicts \eqref{h-derivative-endpt-limit=infty}. Hence no such sequence $\{r_i\}_{i=1}^{\infty}$ exists. Thus there exists a constant $c_3>0$ such that \eqref{h-derivative-uniform-upper-bd} holds and the lemma follows.

\end{proof}

\begin{lem}\label{h-derivative-lower-bd-lem}
Let $\lambda\in\mathbb{R}$ and $n\ge 2$. Suppose $h\in C^2((0,L))$ satisfies \eqref{h-eqn}
in $(0,L)$ for some constant $L\in (0,\infty)$ such that $L<-(n-1)/\lambda$ if $\lambda<0$. Then there exists a constant $c_4\in\mathbb{R}$ such that
\begin{equation}\label{h-derivative-uniform-lower-bd}
h_r(r)\ge c_4\quad\forall L/2\le r\le L.
\end{equation}  
\end{lem}
\begin{proof}
Let $c_1>0$, $c_2>0$, $c_3>0$, be as given by Lemma \ref{h-lower-bd-lem}, Lemma \ref{h-upper-bd-lem} and Lemma \ref{h-derivative-upper-bd-lem} respectively. 
Suppose \eqref{h-derivative-uniform-lower-bd}  does not hold. Then there exists a sequence $\{r_i\}_{i=1}^{\infty}\subset (L/2,L)$, $r_i<r_{i+1}$ for all $i\in\Z^+$, $r_i\to L$ as $i\to\infty$, such that
\begin{equation}\label{h-derivative-endpt-limit=neg-infty}
h_r(r_i)<0\quad\forall i\in\Z^+\quad\mbox{ and }\quad h_r(r_i)\to -\infty\quad\mbox{ as }i\to\infty.
\end{equation}
Let $C_4$ be given by \eqref{c4-defn}. By \eqref{h-derivative-endpt-limit=neg-infty} there exists $i_0\in\Z^+$ such that
\begin{equation}\label{hr-upper-bd6}
h_r(r_i)<-C_4\quad\forall i\ge i_0.
\end{equation}
Then by \eqref{hr-upper-bd6} and an argument similar to the proof of Lemma \ref{h-derivative-upper-bd-lem} we get that \eqref{hrr-lower-bd-ineqn} holds. Hence by \eqref{hrr-lower-bd-ineqn} and continuity there exists $0<\delta<L-r_{i_0}$ such that \eqref{hrr-hr-hr-ineqn} holds. Then by \eqref{hrr-hr-hr-ineqn},
\begin{equation*}
h_{rr}(r)>0\quad\forall r_{i_0}\le r<r_{i_0}+\delta\quad\Rightarrow\quad h_r(r)>h_r(r_{i_0})\quad\forall r_{i_0}<r<r_{i_0}+\delta.
\end{equation*}
Let 
\begin{equation*}
b_2=\sup\left\{b\in (r_{i_0},L):h_r(r)>h_r(r_{i_0})\quad\forall r_{i_0}\le r<b\right\}.
\end{equation*}
Then $b_2\ge r_{i_0}+\delta$. Suppose $b_2<L$. Then 
\begin{align}
&h_r(r)>h_r(r_{i_0})\quad\forall r_{i_0}\le r<b\quad\mbox{ and }\quad  h_r(b_2)=h_r(r_{i_0})<-C_4\label{hr-eqn5}\\
\Rightarrow\quad&h_{rr}(b_2)\le 0.\label{hr-eqn11}
\end{align}
On the other hand by \eqref{hr-eqn5} and an argument similar to the proof of Lemma \ref{h-derivative-upper-bd-lem} we have
\begin{equation*}
h_{rr}(b_2)>\frac{h_r(b_2)^2}{4h(b_2)}
\end{equation*}
which contradicts \eqref{hr-eqn11}. Hence $b_2=L$. Thus
\begin{equation*}
h_r(r)>h_r(r_{i_0})\quad\forall r_{i_0}\le r<L
\end{equation*}
which contradicts \eqref{h-derivative-endpt-limit=neg-infty}.
Hence no such sequence $\{r_i\}_{i=1}^{\infty}$ exists. Thus there exists a constant $c_4>0$ such that \eqref{h-derivative-uniform-lower-bd} holds and the lemma follows.

\end{proof}

We next observe that by standard ODE theory we have the following result.

\begin{lem}\label{h-extension-lem}
Let $\lambda\ge 0$, $\mu_1\in\mathbb{R}$, $L>0$, $n\ge 2$, $b_1\in (c_1,c_2)$, $b_2\in (c_4,c_3)$ for some constants $c_2>c_1>0$ and $c_3>c_4$. Then there exists a constant $0<\delta_1<L/4$ depending only on $c_1, c_2, c_3, c_4$ such that for any $r_0\in (L/2,L)$ \eqref{h-eqn} has a unique solution $\4{h}\in C^2((r_0-\delta_1,r_0+\delta_1))$ in $(r_0-\delta_1,r_0+\delta_1)$ which satisfies
\begin{equation}\label{h-tilde-initial-condition}
\4{h}(r_0)=b_1\quad\mbox{ and }\quad \4{h}_r(r_0)=b_2.
\end{equation} 
\end{lem} 

We are now ready for the proof of Theorem \ref{h-existence-thm}.

\noindent{\bf Proof of Theorem \ref{h-existence-thm}}: 
Uniqueness of solution of \eqref{h-ode-initial-value-problem} follows from the local uniqueness of solution
of \eqref{h-eqn} in $(0,\3)$ for some small $\3>0$ which satisfies \eqref{h-initial-condition} and from the uniqueness of solution of ODE of \eqref{h-eqn} on $[\3/2,\infty)$ with given values of $h(\3/2)$ and $h_r(\3/2)$.

We next observe that by Lemma \ref{h-local-existence-lem} there exists a constant $\3>0$ such that 
\eqref{h-eqn} has a unique solution $h\in C^2((0,\3))\cap C^1([0,\3))$ in $(0,\3)$ which satisfies
\eqref{h-initial-condition}. Let $[0,L)$ be the maximal interval of existence of solution $h\in C^2((0,L))\cap C^1([0,L))$ of \eqref{h-eqn} in $(0,L)$ which satisfies \eqref{h-initial-condition}. 
Suppose $L<\infty$. Then by Lemma \ref{h-lower-bd-lem}, Lemma \ref{h-upper-bd-lem}, Lemma \ref{h-derivative-upper-bd-lem} and Lemma \ref{h-derivative-lower-bd-lem} there exists constants 
$c_2>c_1>0$ and $c_3>c_4$ such that \eqref{h-uniform-lower-bd}, \eqref{h-uniform-upper-bd}, \eqref{h-derivative-uniform-upper-bd} and \eqref{h-derivative-uniform-lower-bd} hold. 

 Then by Lemma \ref{h-extension-lem}  there exists a constant $0<\delta_1<L$ depending only on $c_1, c_2, c_3, c_4$ such that for any $r_0\in (L/2,L)$ \eqref{h-eqn} has a unique solution $\4{h}\in C^2((r_0-\delta_1,r_0+\delta_1))$ in $(r_0-\delta_1,r_0+\delta_1)$ which satisfies
\eqref{h-tilde-initial-condition} with $b_1=h(r_0)$ and $b_1=h_r(r_0)$. We now set $r_0=L-(\delta_1/2)$ and extend $h$ to a function on $[0,L+(\delta_1/2))$ by setting
$h(r)=\4{h}(r)$ for any $r\in [L,L+(\delta_1/2))$. Then $h\in C^2((0,L+(\delta_1/2)))\cap C^1([0,L+(\delta_1/2)))$ is a solution of \eqref{h-eqn} in $(0,L+\delta_1)$ which satisfies \eqref{h-initial-condition}. This contradicts the choice of $L$. Hence $L=\infty$ and there exists a unique solution $h\in C^2((0,\infty))\cap C^1([0,\infty))$ of \eqref{h-ode-initial-value-problem}.

{\hfill$\square$\vspace{6pt}}

By  Theorem \ref{h-existence-thm} and Lemma \ref{h-monotone-lem} we get 
Theorem \ref{h-steady-soln-property-thm1}.
By Lemma \ref{h-lower-bd-lem}, Lemma \ref{h-upper-bd-lem}, Lemma \ref{h-derivative-upper-bd-lem} and Lemma \ref{h-derivative-lower-bd-lem} and a proof similar to the proof of Theorem \ref{h-existence-thm} we get Theorem \ref{h-lambda-neg-existence-thm}.

By scaling and the uniqueness result of Theorem \ref{h-existence-thm} we have the following corollary.

\begin{cor}
Let $2\le n\le 4$. For any $\mu_1\in\mathbb{R}$ let $h(r;\mu_1)\in C^2((0,\infty))\cap C^1([0,\infty))$ be the unique solution of \eqref{h-steady-soliton-eqn}. Then
\begin{equation*}
h(\mu r;\mu_1)=h(r;\mu\mu_1)\quad\forall \mu>0, \mu_1\in\mathbb{R}.
\end{equation*}

\end{cor}

\section{Existence of analytic solution}
\setcounter{equation}{0}
\setcounter{thm}{0}

In this section we will use a modification of the power series method of \cite{Br} to proof the existence of unique analytic solutions of \eqref{h-ode-initial-value-problem} for any $n\ge 2$, $\lambda\ge 0$ and $\mu_1\in\Rr$.

\begin{lem}\label{h-local-existence-analytic-soln-lem}
Let $\lambda, \mu_1\in\mathbb{R}$ and $n\ge 2$. Then there exists a constant $\3>0$ such that 
\eqref{h-eqn} has a unique analytic solution $h$ in $[0,\3)$ which satisfies
\eqref{h-initial-condition}.
\end{lem} 
\begin{proof}
Since local analytic solution of \eqref{h-eqn}, \eqref{h-initial-condition}, about $r=0$ is unique, we only need to proof existence of local analytic solution of \eqref{h-eqn}, \eqref{h-initial-condition}, about $r=0$. We will use a modification of the proof of Proposition 1 of \cite{Br} to proof this lemma. Since the proof is similar to the proof of Proposition 1 of \cite{Br} we will only sketch the argument here. Suppose the solution $h$ of \eqref{h-eqn}, \eqref{h-initial-condition}, near the origin is of the form,
\begin{equation}\label{h-power-series}
h(r)=\sum_{i=0}^{\infty}c_ir^i
\end{equation}
for some constants $c_1,c_2,\dots$. Substituting \eqref{h-power-series} into \eqref{h-eqn}, \eqref{h-initial-condition}, and equating the coefficients of $r^i$, $i\in\Z^+$, we get that \eqref{h-power-series} satisfies \eqref{h-eqn} and \eqref{h-initial-condition} if and only if
\begin{equation}\label{c0-c1}
c_0=1,\quad c_1=\mu_1,
\end{equation}
and
\begin{equation}\label{ci-recurrence-relation}
(k-1)(2k+n-1)c_k+\lambda (k-1)c_{k-1}+\sum_{1\le j\le k-1}(3j^2-(2+k)j-(n-1))c_jc_{k-j}=0\quad\forall k\ge 2.
\end{equation}
Let $c_k$, $k\ge 2$, be given uniquely in terms of $c_1,\dots, c_{k-1}$, by the recurrence relation \eqref{ci-recurrence-relation}.
Note that
\begin{align}\label{term-by-term-comparison}
|3j^2-(2+k)j-(n-1)|\le&(3j-2)j+(2k+n-1)j\qquad\qquad\qquad\quad\,\,\,\forall 1\le j\le k-1,k\ge 2\notag\\
\le&(3(k-1)-2)(k-1)+(2k+n-1)(k-1)\quad\forall 1\le j\le k-1,k\ge 2\notag\\
\le&\frac{3}{2}(2k+n-1)(k-1)+(2k+n-1)(k-1)\quad\forall 1\le j\le k-1,k\ge 2\notag\\
\le&\frac{5}{2}(2k+n-1)(k-1)\qquad\qquad\qquad\qquad\quad\,\,\forall 1\le j\le k-1,k\ge 2.
\end{align}
Hence by \eqref{ci-recurrence-relation} and \eqref{term-by-term-comparison},
\begin{align}\label{ci-ineqn}
&(k-1)(2k+n-1)|c_k|\le|\lambda|(k-1)|c_{k-1}|+\frac{5}{2}(2k+n-1)(k-1)\sum_{1\le j\le k-1}|c_j||c_{k-j}|
\quad\forall k\ge 2\notag\\
\Rightarrow\quad&|c_k|\le|\lambda||c_{k-1}|+\frac{5}{2}\sum_{1\le j\le k-1}|c_j||c_{k-j}|
\quad\forall k\ge 2.
\end{align}
Let
\begin{equation}\label{b-power-series}
b(r)=\sum_{i=1}^{\infty}c_i'r^i
\end{equation}
where 
\begin{equation}\label{ck'-defn}
c_1'=|\mu_1|\quad\mbox{ and }\quad c_k'=|\lambda|c_{k-1}'+\frac{5}{2}\sum_{1\le j\le k-1}c_j'c_{k-j}'
\quad\forall k\ge 2.
\end{equation}
By \eqref{c0-c1}, \eqref{ci-ineqn}, \eqref{ck'-defn} and an induction argument on the sequences $\{c_k\}_{k=1}^{\infty}$, $\{c_k'\}_{k=1}^{\infty}$, we get,
\begin{equation}\label{ck-ck'-ineqn}
|c_k|\le c_k'\quad\forall k\ge 1.
\end{equation}
Hence the series \eqref{h-power-series} will be convergent if the series  \eqref{b-power-series} is convergent. Now by \eqref{ck'-defn} $b(r)$ satisfies
\begin{align}\label{b-eqn}
&b(r)=|\lambda|rb(r)+\frac{5}{2}b(r)^2+c_1'r\notag\\
\Leftrightarrow\quad&\frac{5}{2}b(r)^2-(1-|\lambda|r)b(r)+c_1'r=0.
\end{align}
Note that \eqref{b-eqn} has a explicit solution of the form,
\begin{equation}\label{b-explicit-form}
b(r)=\frac{1-|\lambda|r-\sqrt{(1-|\lambda|r)^2-10c_1'r}}{5}\quad\forall |r|\le\3
\end{equation}
where
\begin{equation*}
\3=\min\left(\frac{1}{10(|\lambda|+1)},\frac{1}{200(c_1'+1)}\right)
\end{equation*}
and the function given by \eqref{b-explicit-form} has a convergent Taylor series expansion on $[0,\3)$
with radius of convergence about $r=0$ greater than or equal to $\3$. By \eqref{ck-ck'-ineqn} the series \eqref{h-power-series} for $h$ also has radius of convergence about $r=0$ greater than or equal to $\3$. Hence the lemma follows.
 
\end{proof}

By Lemma \ref{h-local-existence-analytic-soln-lem} and an argument similar to the proof of Theorem \ref{h-existence-thm} Theorem \ref{h-existence-analytic-soln-thm} follows.
By scaling and the uniqueness result of Theorem \ref{h-existence-analytic-soln-thm} we have the following corollary.

\begin{cor}
Let $n\ge 2$. For any $\mu_1\in\mathbb{R}$ let $h(r;\mu_1)$ be the unique analytic solution of \eqref{h-steady-soliton-eqn} on $[0,\infty)$. Then
\begin{equation*}
h(\mu r;\mu_1)=h(r;\mu\mu_1)\quad\forall \mu>0, \mu_1\in\mathbb{R}.
\end{equation*}

\end{cor}

\section{Asymptotic behaviour of rotationally symmetric steady gradient Ricci soliton}
\setcounter{equation}{0}
\setcounter{thm}{0}

In this section we will prove the asymptotic behaviour of the solution $h$ of \eqref{h-steady-soliton-eqn}
as $r\to\infty$. We will assume that $\lambda=0$, $\mu_1<0$, $n\ge 2$ and $h\in C^2((0,\infty))\cap C^1([0,\infty))$ is a solution of \eqref{h-steady-soliton-eqn}  in this section.
By Theorem \ref{h-steady-soln-property-thm1} \eqref{hr-sign} holds. Hence
\begin{equation}\label{h-limit-at-infty}
c_0=\lim_{r\to\infty} h(r)\in [0,1)\quad\mbox{ exists}
\end{equation}
and
\begin{equation}\label{h-range}
c_0<h(r)<1\quad\forall r>0.
\end{equation}
We will prove that $c_0=0$. We now let
\begin{equation}\label{q-defn}
q(r)=\frac{rh_r(r)}{h(r)}\quad\forall r\ge 0.
\end{equation}
Then by \eqref{hr-sign},
\begin{equation}\label{q-neg}
q(r)<0\quad\forall r>0.
\end{equation} 
Note that by \eqref{h-steady-soliton-eqn} and a direct computation $q$ satisfies
\begin{equation}\label{qr-eqn10}
q_r(r)=\frac{q(r)}{r}\left(1-\frac{n-1}{2h(r)}\right)-\left(\frac{q(r)^2}{2r}+(n-1)\frac{1-h(r)}{2rh(r)}\right)=-\frac{H(q(r),r)}{2r}\quad\forall r>0
\end{equation}
where
\begin{equation}\label{H-b-r-defn}
H(b,r)=\frac{-4nh(r)^2+8(n-1)h(r)-(n-1)^2}{4h(r)^2}+\left(b+\frac{n-1-2h(r)}{2h(r)}\right)^2\quad\forall r>0, b\in\mathbb{R}.
\end{equation}

\begin{lem}\label{h-infty-to-0-lem}
Let $\mu_1<0$ and $n\ge 2$. Let $h\in C^2((0,\infty))\cap C^1([0,\infty))$ be  a solution  of
\eqref{h-steady-soliton-eqn} and $c_0$, $q$, be given by \eqref{h-limit-at-infty} and \eqref{q-defn} respectively.  If there exist constants $r_0>0$ and $C>0$ such that
\begin{equation}\label{q-uniform-upper-bd10}
q(r)\le -C\quad\forall r\ge r_0,
\end{equation}
then $c_0=0$.
\end{lem}
\begin{proof}
By \eqref{q-uniform-upper-bd10},
\begin{align*}
\frac{rh_r(r)}{h(r)}\le -C\quad\forall r\ge r_0\quad
\Rightarrow\quad&\log(h(r)/h(r_0))\le -C\log(r/r_0)\quad\forall r\ge r_0\notag\\
\Rightarrow\quad&0\le c_0<h(r)\le\frac{r_0^Ch(r_0)}{r^C}\quad\forall r\ge r_0\notag\\
\Rightarrow\quad&0\le c_0<h(r)\to 0\quad\mbox{ as }r\to\infty
\end{align*}
and the lemma follows.
\end{proof}

We are now ready to prove Theorem \ref{steady-soln-asymptotic-behaviour-thm}.

\noindent{\bf Proof of Theorem \ref{steady-soln-asymptotic-behaviour-thm}}: Suppose
\begin{equation}\label{c0-not=0}
0<c_0<1.
\end{equation}
Then by \eqref{h-limit-at-infty},
\begin{equation}\label{H-limit}
\underset{\substack{r\to\infty\\b\to 0}}{\lim}H(b,r)
=\frac{-4nc_0^2+8(n-1)c_0-(n-1)^2}{4c_0^2}+\frac{(n-1-2c_0)^2}{4c_0^2}
=\frac{(n-1)c_0(1-c_0)}{c_0^2}=:C_5>0 \,\mbox{(say)}.
\end{equation}
By \eqref{H-limit} there exists $r_0'>0$ and $b_0>0$ such that
\begin{equation}\label{H-uniformly-upper-lower-bd-at-infty}
C_5/2<H(b,r)<3C_5/2\quad\forall r\ge r_0', |b|\le b_0.
\end{equation}
We now claim that there exist constants $C>0$ and $r_0>0$ such that \eqref{q-uniform-upper-bd10} holds. Suppose the claim does not hold. Then there exists a sequence $\{r_i\}_{i=1}^{\infty}\subset (r_0',\infty)$, $r_i<r_{i+1}$ for any $i\in\mathbb{Z}^+$ and $r_i\to\infty$ as $i\to\infty$, such that 
\begin{equation*}
q(r_i)\to 0\quad\mbox{ as }i\to\infty.
\end{equation*}
Without loss of generality we may assume that
\begin{equation}\label{q-sequence-bd-below}
0>q(r_i)>-b_0\quad\forall i\in\mathbb{Z}^+.
\end{equation} 
Suppose there exists $i_0\in\mathbb{Z}^+$ and $s_0\in (r_{i_0},r_{i_0+1})$ such that
$q(s_0)<-b_0$. Let
\begin{equation}\label{s1-defn}
s_1=\sup\{s'>s_0:q(s)<-b_0\quad\forall s_0\le s<s'\}.
\end{equation}
By \eqref{q-sequence-bd-below} and \eqref{s1-defn}, $s_0<s_1<r_{i_0+1}$,
\begin{align}
&q(s)<-b_0\quad\forall s_0\le s<s_1\quad\mbox{ and }\quad q(s_1)=-b_0\label{q=b0}\\
\Rightarrow\quad&q_r(s_1)\ge 0.\label{qr-positive3}
\end{align}
On the other hand by \eqref{qr-eqn10}, \eqref{H-uniformly-upper-lower-bd-at-infty} and \eqref{q=b0},
\begin{equation*}
q_r(s_1)=-\frac{H(q_r(s_1), s_1)}{2s_1}\le-\frac{C_5}{4s_1} <0
\end{equation*}
which contradicts \eqref{qr-positive3}. Thus no such constants $i_0$ and $s_0$ exist
and
\begin{equation}\label{q-uniformly-bd-below1}
0>q(r)\ge -b_0\quad\forall r\ge r_1.
\end{equation}
Then by \eqref{qr-eqn10}, \eqref{H-uniformly-upper-lower-bd-at-infty} and \eqref{q-uniformly-bd-below1},
\begin{align*}
&q_r(r)=-\frac{H(q_r(r), r)}{2r}\le-\frac{C_5}{4r} <0\quad\forall r\ge r_1\notag\\
\Rightarrow\quad&q(r)\le q(r_1)\quad\forall r\ge r_1.
\end{align*}
This contradicts the assumption that there do not exist any constants $C>0$ and $r_0>0$ such that \eqref{q-uniform-upper-bd10} holds. Hence there exist constants $C>0$ and $r_0>0$ such that \eqref{q-uniform-upper-bd10} holds. Then by Lemma \ref{h-infty-to-0-lem} $c_0=0$ and the theorem follows.

{\hfill$\square$\vspace{6pt}}

We now let 
\begin{equation}\label{u-defn}
u(r)=rh(r).
\end{equation}
and
\begin{equation}\label{p-defn}
p(r)=rh_r(r)\quad\forall r\ge 0.
\end{equation}
Then by \eqref{hr-sign},
\begin{equation}\label{p-u-neg}
p(r)<0\quad\forall r>0\quad\Rightarrow\quad u_r(r)=rh_r(r)+h(r)=p(r)+h(r)<h(r)<1\quad\forall r>0.
\end{equation}
By \eqref{h-steady-soliton-eqn} and a direct computation $p$ satisfies
\begin{equation}\label{p-eqn}
p_r(r)=\frac{(p(r)+h(r))[p(r)-(n-1)(1-h(r))]-(n-2)h(r)p(r)}{2u(r)}\quad\forall r>0.
\end{equation}
By \eqref{p-u-neg}, \eqref{p-eqn} and a direct computation  $u$ satisfies
\begin{equation}\label{u-eqn}
u_{rr}(r)=\frac{-u_r(r)(n-1-2h(r)-u_r(r))+(n-4)h(r)^2}{2rh(r)}\quad\forall r>0.
\end{equation}

\begin{lem}\label{ur-positive-negative-lem}
Let $\mu_1<0$ and $n\ge 2$. Let $h\in C^2((0,\infty))\cap C^1([0,\infty))$ be  a solution  of
\eqref{h-steady-soliton-eqn} and $u$ be given by  \eqref{u-defn}.  Then for any $n\ge 4$, 
\begin{equation}\label{ur>0}
u_r(r)>0\quad\forall r>0.
\end{equation}
When $2\le n<4$, then either \eqref{ur>0} holds 
or there exists a  constant $s_0>0$ such that
\begin{equation}\label{ur<0}
\left\{\begin{aligned}
&u_r(r)>0\quad\forall 0<r<s_0\\
&u_r(s_0)=0\\
&u_r(r)<0\quad\forall r>s_0\end{aligned}\right.
\end{equation}
holds. Moreover for any $n\ge 2$ when \eqref{ur>0} holds,
\begin{equation}\label{ur-goes-to-0}
\lim_{r\to\infty}u_r(r)=0.
\end{equation}
\end{lem}
\begin{proof}
Since $u_r(0)=h(0)=1$, there exists a constant $\delta_1>0$ such that $u_r(r)>0$ for any $0\le r<\delta_1$. Let
\begin{equation*}
s_0=\sup\{s'>0:u_r(r)>0\quad\forall 0\le r<s'\}.
\end{equation*}
Then $s_0\ge\delta_1$. If $s_0=\infty$, then \eqref{ur>0} holds. Suppose $s_0<\infty$. Then 
\begin{align}\label{urr-neg}
&u_r(r)>0\quad\forall 0<r<s_0\quad\mbox{ and }\quad u_r(s_0)=0\notag\\
\Rightarrow\quad&u_{rr}(s_0)\le 0.
\end{align} 
We now divide the proof into three cases.

\noindent $\underline{\text{\bf Case 1}}$: $n>4$.

\noindent  By \eqref{u-eqn},
\begin{equation}
u_{rr}(s_0)=\frac{(n-4)h(s_0)}{2s_0}>0.
\end{equation}
which contradicts \eqref{urr-neg}. Hence $s_0=\infty$ and \eqref{ur>0} holds. 

\noindent $\underline{\text{\bf Case 2}}$: $n=4$.

\noindent By \eqref{u-eqn} both the functions $u_r$ and
\begin{equation*}
\widetilde{w}(r)=0\quad\forall r\ge 0
\end{equation*}
satisfy 
\begin{equation*}
\widetilde{w}_{r}(r)=-\frac{\widetilde{w}(r)(n-1-2h(r)-\widetilde{w}(r))}{2rh(r)}\quad\forall r>0.
\end{equation*}
and $\widetilde{w}(s_0)=0=u_r(s_0)$. Thus by uniqueness of ODE, 
\begin{equation*}
u_r(r)=\widetilde{w}(r)=0\quad\forall r>0\quad\Rightarrow\quad 1=u_r(0)=\lim_{r\to 0}u_r(r)=0
\end{equation*}
and contradiction arises. Hence $s_0=\infty$ and \eqref{ur>0} holds.

\noindent $\underline{\text{\bf Case 3}}$: $2\le n<4$.

\noindent By \eqref{u-eqn},
\begin{equation*}
u_{rr}(s_0)=\frac{(n-4)h(s_0)}{2s_0}<0.
\end{equation*}
Then there exists a constant $s_1>s_0$ such that
\begin{align*}
&u_{rr}(r)<0\quad\forall s_0\le r<s_1\notag\\
\Rightarrow\quad&u_r(r)<u_r(s_0)=0\quad\forall s_0<r<s_1.
\end{align*}
Let
\begin{equation*}
s_2=\sup\{s'>0:u_r(r)<0\quad\forall s_0<r<s'\}.
\end{equation*}
Then $s_2\ge s_1$. If $s_2<\infty$, then 
\begin{align}
&u_r(r)<0\quad\forall s_0<r<s_2 \quad\mbox{ and }\quad u_r(s_2)=0\label{ur-zero5}\\
\Rightarrow\quad&u_{rr}(s_2)\ge 0.\label{urr-postive}
\end{align}
By \eqref{u-eqn} and \eqref{ur-zero5},
\begin{equation*}
u_{rr}(s_2)=\frac{(n-4)h(s_2)}{2s_2}<0\quad\forall 2\le n<4
\end{equation*}
which contradicts \eqref{urr-postive}. Hence $s_2=\infty$ and \eqref{ur<0} holds with 
$s_0\in (0,\infty)$. Finally for any $n\ge 2$ when  \eqref{ur>0} holds, by \eqref{p-u-neg}, \eqref{ur>0} and Theorem \ref{steady-soln-asymptotic-behaviour-thm} we get \eqref{ur-goes-to-0} and the lemma follows.
\end{proof}

By \eqref{hr-sign}, \eqref{p-u-neg} and Lemma \ref{ur-positive-negative-lem} we have the following lemma.

\begin{lem}\label{q-bded-lem}
Let $\mu_1<0$ and $n\ge 2$. Suppose $h\in C^2((0,\infty))\cap C^1([0,\infty))$ is a solution  of \eqref{h-steady-soliton-eqn}. Let $q$ be given by \eqref{q-defn}. Suppose \eqref{ur>0} holds. 
Then 
\begin{equation*}\label{q-bded-eqn}
-1<q(r)<0\quad\forall r>0.
\end{equation*}
\end{lem}

Note that by \eqref{h-infty=0} there exists a constant $r_0>1$ such that 
\begin{equation}\label{h-upper-bd2}
0<h(r)<1/500\quad\forall r>r_0.
\end{equation} 

\begin{lem}\label{rh-limit-lem0}
Let $\mu_1<0$ and $2\le n\le 4$. Suppose $h\in C^2((0,\infty))\cap C^1([0,\infty))$ is  a solution  of \eqref{h-steady-soliton-eqn}. Then \eqref{r-h-infty-limit} holds.
\end{lem}
\begin{proof}
Let $u$ be given by \eqref{u-defn}. 
By Lemma \ref{ur-positive-negative-lem} we can divide the proof into two cases.

\noindent $\underline{\text{\bf Case 1}}$:  \eqref{ur>0} holds.

\noindent By \eqref{p-u-neg}, \eqref{u-eqn}, \eqref{ur>0} and \eqref{h-upper-bd2},
\begin{align}\label{u-upper-bd2}
&u_{rr}(r)\le -\frac{u_r(r)(n-(101/100)-u_r(r))}{2u(r)}\qquad\quad\forall r>r_0\notag\\
\Rightarrow\quad&\frac{u_{rr}}{n-(101/100)-u_r(r)}\le-\frac{u_r(r)}{2u(r)}\qquad\qquad\qquad\,\,\forall r>r_0\notag\\
\Rightarrow\quad&\log\left(\frac{n-(101/100)-u_r(r_0)}{n-(101/100)-u_r(r)}\right)\le-\frac{1}{2}\log\left(\frac{u(r)}{u(r_0)}\right)\quad\forall r>r_0\notag\\
\Rightarrow\quad&u(r)\le u(r_0)\left(\frac{n-(101/100)-u_r(r)}{n-(101/100)-u_r(r_0)}\right)^2\le u(r_0)\left(\frac{n-(101/100)}{n-(203/200)}\right)^2\quad\forall r>r_0.
\end{align}
By \eqref{ur>0} and \eqref{u-upper-bd2} $u(r)$ is monotone increasing and uniformly bounded 
in  $r>r_0$. Hence $0<b_1=\lim_{r\to\infty}u(r)\in [0,\infty)$ exists and  \eqref{r-h-infty-limit} follows.

\noindent $\underline{\text{\bf Case 2}}$:  There exists a  constant $s_0>0$ such that \eqref{ur<0} holds.

\noindent By \eqref{ur<0}  $u(r)>0$ is monotone decreasing in $r>s_0$. Hence  \eqref{r-h-infty-limit}  follows.

\end{proof}

\begin{lem}\label{ur-bded-lem}
Let $\mu_1<0$ and $2\le n<4$. Suppose $h\in C^2((0,\infty))\cap C^1([0,\infty))$ is  a solution  of \eqref{h-steady-soliton-eqn}. Let $u$ be given by \eqref{u-defn}. Suppose there exists a  constant $s_0>0$ such that \eqref{ur<0} holds. Then \eqref{ur-goes-to-0} holds.
\end{lem}
\begin{proof}
Let $2\le n<4$. Suppose  \eqref{ur-goes-to-0} does not hold. Then there exists a constant $0<\delta<1$ and a sequence 
$\{\overline{s}_i\}_{i=1}^{\infty}$, $\overline{s}_i\to\infty$ as $i\to\infty$, such that
\begin{equation}\label{ur-away-0}
u_r(\overline{s}_i)<-\delta\quad\forall i\in\Z^+.
\end{equation}
 By \eqref{h-infty=0} there exists a constant $s_2>s_0$ such that
\begin{equation}\label{h-small-eqn}
0<h(r)\le\delta/300\quad\forall r\ge s_2.
\end{equation}
By Lemma \ref{rh-limit-lem0} \eqref{r-h-infty-limit} holds. Then by \eqref{r-h-infty-limit} and \eqref{ur<0}, there exists a constant $s_3>\max (s_2,1/\delta)$ such that
\begin{equation}\label{u-difference-small}
0<u(s_3)-u(2s_3)<1.
\end{equation}
Hence by \eqref{u-difference-small} and the mean value theorem there exists a constant $s_4\in (s_3,2s_3)$ such that
\begin{equation}\label{ur-approx-0}
u_r(s_4)=\frac{u(2s_3)-u(s_3)}{s_3}>-\frac{1}{s_3}\ge -\delta.
\end{equation}
Then by \eqref{ur-away-0} there exists $i_0\in\Z^+$ such that $s_5:=\overline{s}_{i_0}>s_4$ and
\begin{equation}
u_r(s_5)<-\delta.
\end{equation}
Let
\begin{equation*}
s_6=\inf\{s'>0:u_r(r)<-\delta\quad\forall s'<r\le s_5\}.
\end{equation*}
Then by \eqref{ur-approx-0} $s_6\in (s_4,s_5)$ and
\begin{align}
&u_r(r)<-\delta\quad\forall s_6<r\le s_5 \quad\mbox{ and }\quad u_r(s_6)=-\delta \label{ur-small2}\\
\Rightarrow\quad&u_{rr}(s_6)\le 0.\label{ur-neg-at-zero}
\end{align}
By \eqref{h-small-eqn} and \eqref{ur-small2},
\begin{equation}\label{ur-squqre-lower-bd}
u_r(r)^2\ge\delta^2\ge 100h(r)^2>2(4-n)h(r)^2\quad\forall s_6\le r\le s_5, 2\le n<4.
\end{equation}
Then by \eqref{u-eqn}, \eqref{ur<0}, \eqref{h-small-eqn} and \eqref{ur-squqre-lower-bd},
\begin{align*}\label{u-ineqn6}
u_{rr}(s_6)\ge&\frac{-u_r(s_6)(n-(101/100)-u_r(s_6))-(4-n)h(s_6)^2}{2u(s_6)}\notag\\
\ge&\frac{u_r(s_6)^2-(4-n)h(s_6)^2}{2u(s_6)}\notag\\
\ge&\frac{u_r(s_6)^2}{4u(s_6)}\\
=&\frac{\delta^2}{4u(s_6)}>0\quad\forall  2\le n<4
\end{align*}
which contradicts \eqref{ur-neg-at-zero}. Hence \eqref{ur-goes-to-0} holds and the lemma follows.

\end{proof}

\begin{lem}\label{q-bded-lem2}
Let $\mu_1<0$ and $2\le n<4$. Suppose $h\in C^2((0,\infty))\cap C^1([0,\infty))$ is  a solution  of \eqref{h-steady-soliton-eqn}. Let $q$ be given by \eqref{q-defn}. Suppose there exists a  constant $s_0>0$ such that \eqref{ur<0} holds. Then there exists a constant $C>0$ such that
\begin{equation}\label{q-bded-eqn2}
-C<q(r)<0\quad\forall r>0.
\end{equation}
\end{lem}
\begin{proof}
Let $2\le n<4$ and $r_0>1$ be as in \eqref{h-upper-bd2}. By Lemma \ref{ur-bded-lem} there exists a constant $s_1>\max(s_0, r_0)$ such that 
\begin{equation}\label{ur-lower-bd}
-1/10<u_r(r)<0\quad\forall r>s_1.
\end{equation} 
Suppose  \eqref{q-bded-eqn2} does not hold for any $C>0$. Then there exists a constant $s_2>s_1$ such that $q(s_2)<-32$. Let
\begin{equation*}
s_3=\sup\{s'>s_2:q(r)<-32\quad\forall s_2\le r<s'\}.
\end{equation*}
Then
\begin{equation}\label{q<-32codition}
q(r)<-32\quad\forall s_2\le r<s_3.
\end{equation}
By \eqref{p-u-neg}, \eqref{h-upper-bd2} and \eqref{ur-lower-bd},
\begin{equation}\label{rhr-lower-bd2}
0>rh_r(r)>-h(r)-\frac{1}{10}\ge -\frac{11}{100}\quad\forall r>s_1.
\end{equation}
By \eqref{qr-eqn10}, \eqref{h-upper-bd2} and \eqref{rhr-lower-bd2},
\begin{align}\label{qr>ineqn1}
rq_r(r)=&-q(r)\left(\frac{n-1+rh_r(r)}{2h(r)}-1\right)-(n-1)\frac{1-h(r)}{2h(r)}\notag\\
\ge&-q(r)\left(\frac{n-(111/100)}{2h(r)}-\frac{1}{100h(r)}\right)-\frac{n-1}{2h(r)}\notag\\
\ge&-\frac{C_1(q(r)+C_2)}{2h(r)}\quad\forall s_2\le r<s_3, 2\le n<4.
\end{align}
where 
\begin{equation}\label{c1-c2-ineqn}
C_1:=n-\frac{113}{100}>0\quad\mbox{ and }\quad 1<C_2
:=\frac{n-1}{n-\frac{113}{100}}<2\quad\forall 2\le n<4.
\end{equation}
By \eqref{q<-32codition} and \eqref{c1-c2-ineqn},
\begin{equation}\label{q-expression-bd-above}
q(r)+C_2<0\quad\forall s_2\le r<s_3, 2\le n<4.
\end{equation}
Suppose $s_3=\infty$. Then by \eqref{qr>ineqn1} and \eqref{q-expression-bd-above},
\begin{align}\label{qr>ineqn2}
&\frac{q_r(r)}{q(r)+C_2}\le -\frac{C_1}{2rh(r)}\quad\forall r\ge s_2, 2\le n<4\notag\\
\Rightarrow\quad&\log\left(\frac{q(r)+C_2}{q(s_2)+C_2}\right)\le -\int_{s_2}^r\frac{C_1}{2\rho h(\rho)}\,d\rho\quad\forall r\ge s_2, 2\le n<4\notag\\
\Rightarrow\quad&q(r)+C_2\ge (q(s_2)+C_2)e^{-\int_{s_2}^r\frac{C_1}{2\rho h(\rho)}\,d\rho}\quad\forall r\ge s_2, 2\le n<4.
\end{align}
Letting $r\to\infty$ in \eqref{qr>ineqn2}, by \eqref{h-upper-bd2} and \eqref{c1-c2-ineqn},
\begin{equation*}
\liminf_{r\to\infty}q(r)\ge -C_2>-2.
\end{equation*}
which contradicts \eqref{q<-32codition} and the
assumption that $s_3=\infty$. Hence $s_3<\infty$ and $q(s_3)=-32$. By \eqref{qr>ineqn1} and \eqref{c1-c2-ineqn},
\begin{equation}\label{qr>0ineqn}
q_r(s_3)\ge -\frac{C_1(q(s_3)+C_2)}{2s_3h(s_3)}\ge\frac{15C_1}{s_3h(s_3)}>0.
\end{equation}
By \eqref{qr>0ineqn} there exists a constant $s_4>s_3$ such that
\begin{align}\label{qr>0ineqn2}
&q_r(r)>0\qquad\qquad\quad\forall s_3\le r\le s_4\notag\\
\Rightarrow\quad&q(r)>q(s_3)=-32\quad\forall s_3\le r\le s_4.
\end{align}
Since  \eqref{q-bded-eqn2} does not hold for any $C>0$, there exists a constant $s_5>s_4$ such that
$q(s_5)<-32$. Let
\begin{equation*}
s_6=\inf\{s'<s_5:q(r)<-32\quad\forall s'<r\le s_5\}.
\end{equation*}
By \eqref{qr>0ineqn2} $s_6\in (s_4,s_5)$. Then
\begin{align}
&q(r)<-32\quad\forall s_6<r\le s_5\quad\mbox{ and }\quad q(s_6)=-32\label{q<-32codition2}\\
\Rightarrow\quad&q_r(s_6)\le 0\label{qr-neg6}.
\end{align}
By \eqref{q<-32codition2} and an argument similar to the proof of \eqref{qr>0ineqn},
\begin{equation}\label{qr>0ineqn3}
q_r(s_6)\ge -\frac{C_1(q(s_6)+C_2)}{2s_6h(s_6)}\ge\frac{15C_1}{s_6h(s_6)}>0
\end{equation}
which contradicts \eqref{qr-neg6}. Thus no such constant $s_5$ exists and
\begin{align}
&q(r)\ge -32\quad\forall r\ge s_4\notag\\
\Rightarrow\quad&q(r)>-C_3\quad\forall r>0
\end{align}
where $C_3=\min (32,1+\max_{0\le r\le s_5}|q(r)|)$. This contradicts our assumption that \eqref{q-bded-eqn2} does not hold for any $C>0$. Hence no such constant $s_2$ exists. Thus there exists a constant $C>0$ such that \eqref{q-bded-eqn2} holds and the lemma follows.

\end{proof}

We are now ready to prove Theorem \ref{h-steady-soln-asymptotic-behaviour2-thm}.

\noindent{\bf Proof of Theorem \ref{h-steady-soln-asymptotic-behaviour2-thm}}:
Let $r_0>1$ be as in \eqref{h-upper-bd2}. Let 
\begin{equation}\label{F-defn}
F(r)=\mbox{exp}\,\left(\frac{n-1}{2}\int_{r_0}^r\frac{d\rho}{\rho h(\rho)}\right)\quad\forall r\ge r_0.
\end{equation}
Multiplying \eqref{qr-eqn10} by $F(r)/r$ and integrating over $(r_0,r)$, $r>r_0$,
\begin{align}\label{q-integral-eqn}
&\frac{F(r)q(r)}{r}=r_0^{-1}q(r_0)-\frac{1}{2}\int_{r_0}^r\left(q(\rho)^2+\frac{(n-1)(1-h(\rho))}{h(\rho)}\right)\frac{F(\rho)}{\rho^2}\,d\rho\notag\\
\Rightarrow\quad&q(r)=\frac{r_0^{-1}q(r_0)}{r^{-1}F(r)}-\frac{1}{2r^{-1}F(r)}\int_{r_0}^r\left(q(\rho)^2+\frac{(n-1)(1-h(\rho))}{h(\rho)}\right)\frac{F(\rho)}{\rho^2}\,d\rho\quad\forall r>r_0.
\end{align}
Let $\{r_i\}_{i=1}^{\infty}\subset (r_0,\infty)$ be a sequence such that $r_i\to\infty$ as $i\to\infty$. By Lemma \ref{ur-positive-negative-lem}, Lemma \ref{q-bded-lem} and Lemma \ref{q-bded-lem2} there exists a constant $C>0$ such that \eqref{q-bded-eqn2} holds. Hence the sequence $\{r_i\}_{i=1}^{\infty}$ has a subsequence which we may assume without loss of generality to be the sequence itself that converges to some constant $q_1\in [-C,0]$ as $i\to\infty$.
By \eqref{h-upper-bd2},
\begin{align}
&F(r)\ge\mbox{exp}\,\left(\int_{r_0}^r\frac{d\rho}{2\rho h(\rho)}\right)
\ge  \mbox{exp}\,\left(50\int_{r_0}^r\frac{ d\rho}{\rho}\right)
=r_0^{-50}r^{50}\quad\forall r>r_0\notag\\
\Rightarrow\quad&r^{-2}F(r)\ge r_0^{-50}r^{48}\to\infty\quad\mbox{ as }r\to\infty\label{r-1-F-limit-infty}\\
\Rightarrow\quad&\int_{r_0}^r\left(q(\rho)^2+\frac{(n-1)(1-h(\rho))}{h(\rho)}\right)\frac{F(\rho)}{\rho^2}\,d\rho
\ge\int_{r_0}^r\rho^{-2}F(\rho)\,d\rho\to\infty\quad\mbox{ as }r\to\infty.\label{expression1-infty}
\end{align}
By \eqref{h-infty=0}, \eqref{q-integral-eqn}, \eqref{r-1-F-limit-infty}, \eqref{expression1-infty} and the l'Hospital rule,
\begin{align*}
q_1=\lim_{i\to\infty}q(r_i)=&-\lim_{i\to\infty}\frac{1}{2r_i^{-1}F(r_i)}\int_{r_0}^{r_i}\left(q(\rho)^2+\frac{(n-1)(1-h(\rho))}{h(\rho)}\right)\frac{F(\rho)}{\rho^2}\,d\rho\notag\\
=&-\frac{1}{2}\lim_{i\to\infty}\frac{\left(q(r_i)^2+\frac{(n-1)(1-h(r_i))}{h(r_i)}\right)r_i^{-2}F(r_i)}{-r_i^{-2}F(r_i)+r_i^{-1}F(r_i)(n-1)(2r_ih(r_i))^{-1}}\notag\\
=&\lim_{i\to\infty}\frac{q(r_i)^2h(r_i)+(n-1)(1-h(r_i))}{2h(r_i)-(n-1)}\notag\\
=&-1.
\end{align*}
Since the sequence $\{r_i\}_{i=1}^{\infty}$ is arbitrary, \eqref{q-limit=-1} holds.

{\hfill$\square$\vspace{6pt}}

\begin{lem}\label{w1-limit-lem}
Let $\mu_1<0$ and $n\ge 2$. Suppose $h\in C^2((0,\infty))\cap C^1([0,\infty))$ is  a solution  of \eqref{h-steady-soliton-eqn}. Let 
\begin{equation}\label{w1-defn}
w(r)=\frac{u_r(r)}{h(r)^2}\quad\forall r>0.
\end{equation}
Then 
\begin{equation}\label{w1-limit}
\lim_{r\to\infty}w(r)=\frac{n-4}{n-1}. 
\end{equation}
\end{lem}
\begin{proof}
Let $q$ and $u$ be given by \eqref{q-defn} and \eqref{u-defn} respectively. By \eqref{h-infty=0}, Theorem \ref{h-steady-soln-asymptotic-behaviour2-thm}, Lemma \ref{ur-positive-negative-lem} and Lemma \ref{ur-bded-lem} there exists 
a constant $r_0>1$ such that \eqref{h-upper-bd2}
and
\begin{equation}\label{q-upper-lower-bd-near-infty}
-3/2<q(r)<-1/2\quad\forall r\ge r_0
\end{equation}
and
\begin{equation}\label{ur-bded}
|u_r|\le 1/500 \quad\forall r\ge r_0
\end{equation}
hold. By \eqref{u-eqn} and a direct computation,
\begin{equation}\label{w-eqn}
w_r(r)+\frac{n-1-2h(r)+4q(r)h(r)-u_r(r)}{2rh(r)}w(r)=\frac{n-4}{2rh(r)}\quad\forall r>0.
\end{equation}
Let
\begin{equation}\label{F1-defn}
F_1(r)=\mbox{exp}\,\left(\int_{r_0}^r\frac{n-1-2h(\rho)+4q(\rho)h(\rho)-u_r(\rho)}{2\rho h(\rho)}\,d\rho\right)\quad\forall r\ge {r_0}.
\end{equation}
Multiplying \eqref{w-eqn} by $F_1$ and integrating over $(r_0,r)$, $r>r_0$,
\begin{align}\label{w-integral-representation}
&w(r)F_1(r)=w(r_0)+\frac{n-4}{2}\int_{r_0}^r\frac{F_1(\rho)}{\rho h(\rho)}\,d\rho\notag\\
\Rightarrow\quad&w(r)=\frac{w(r_0)}{F_1(r)}+\frac{n-4}{2F_1(r)}\int_{r_0}^r\frac{F_1(\rho)}{\rho h(\rho)}\,d\rho\quad\forall r\ge r_0.
\end{align}
By \eqref{p-u-neg}, \eqref{h-upper-bd2}, \eqref{q-upper-lower-bd-near-infty} and \eqref{ur-bded},
\begin{align}\label{F1-goes-to-infty}
\frac{n-1-2h(\rho)+4q(\rho)h(\rho)-u_r(\rho)}{2h(\rho)}\ge&\frac{1-\frac{2}{500}-\frac{4}{500}\cdot\frac{3}{2}-\frac{1}{500}}{\frac{2}{500}}> 245\notag\\
\Rightarrow\qquad\qquad\qquad\qquad\qquad\qquad\quad F_1(r)\ge&\mbox{exp}\,\left(245\int_{r_0}^r\frac{d\rho}{\rho}\right)\ge (r/r_0)^{245}\,\,\forall r\ge {r_0}\notag\\
\to&\infty\quad\mbox{ as }r\to\infty.
\end{align}
Then by \eqref{h-upper-bd2} and \eqref{F1-goes-to-infty},
\begin{align}\label{F1-integral-infty}
\int_{r_0}^r\frac{F_1(\rho)}{\rho h(\rho)}\,d\rho\ge&\frac{500}{245r_0^{245}}(r^{245}-r_0^{245})\quad\forall r\ge {r_0}\notag\\
\to&\infty\quad\mbox{ as }r\to\infty.
\end{align}
By Lemma \ref{ur-positive-negative-lem} and Lemma \ref{ur-bded-lem} \eqref{ur-goes-to-0} holds.
By \eqref{h-infty=0}, \eqref{ur-goes-to-0}, \eqref{w-integral-representation}, \eqref{F1-goes-to-infty}, \eqref{F1-integral-infty}, the
l'Hosiptal rule and Theorem \ref{h-steady-soln-asymptotic-behaviour2-thm}, we get
\begin{equation*}
\lim_{r\to\infty}w(r)=\lim_{r\to\infty}\frac{w(r_0)}{F_1(r)}=0\quad\mbox{ if }n=4
\end{equation*}
and
\begin{equation*}
\lim_{r\to\infty}w(r)=\frac{n-4}{2}\lim_{r\to\infty}\frac{F_1(r)(r h(r))^{-1}}{F_1(r)(n-1-2h(r)+4q(r)h(r)-u_r(r))(2r h(r))^{-1}}=\frac{n-4}{n-1}\quad\forall n\ge 2, n\ne 4
\end{equation*}
and \eqref{w1-limit} follows.

\end{proof}

\begin{lem}\label{rh-limit-lem}
Let $\mu_1<0$, $n\ge 4$ and $h\in C^2((0,\infty))\cap C^1([0,\infty))$ be  a solution  of \eqref{h-steady-soliton-eqn}. Then \eqref{r-h-infty-limit} holds
with $b_1>0$.
\end{lem}
\begin{proof}
Let $u$ be given by \eqref{u-defn}. By Lemma \ref{w1-limit-lem}
there exists a constant $s_1>1$ such that
\begin{align}\label{u-upper-bd8}
&\frac{u_r(r)}{h(r)^2}<\frac{n-(3/2)}{n-1}=:C_3<1\quad\forall r\ge s_1\notag\\
\Rightarrow\quad&u(r)^{-2}u_r(r)<C_3r^{-2}\quad\forall r\ge s_1\notag\\
\Rightarrow\quad&\frac{1}{u(s_1)}-\frac{1}{u(r)}<\frac{C_3}{s_1}- \frac{C_3}{r}\quad\forall r\ge s_1\notag\\
\Rightarrow\quad&\frac{1}{u(r)}>\frac{1}{s_1h(s_1)}-\frac{C_3}{s_1}+\frac{C_3}{r}
\ge\frac{1-C_3}{s_1}\quad\forall r\ge s_1\notag\\
\Rightarrow\quad&u(r)<\frac{s_1}{1-C_3}\quad\forall r\ge s_1.
\end{align}
By \eqref{ur>0} and \eqref{u-upper-bd8} $u(r)>0$ is monotone increasing and uniformly bounded 
in  $r>s_1$. Hence $0<b_1=\lim_{r\to\infty}u(r)$ exists and  \eqref{r-h-infty-limit} follows.

\end{proof}

\noindent{\bf Proof of Theorem \ref{steady-soln-asymptotic-hr-behaviour-thm}}:

\noindent \eqref{r-h-infty-limit} is proved in Lemma \ref{rh-limit-lem0} and Lemma \ref{rh-limit-lem}. Hence it remains to prove \eqref{h-2nd-order-limit}.  Let 
\begin{equation*}\label{v-defn}
v(r)=r(q(r)+1)\quad\forall r\ge 0.
\end{equation*}
Then by \eqref{qr-eqn10} and a direct computation,
\begin{equation}
v_r(r)+\left(\frac{n-1}{h(r)}-5+q(r)\right)\frac{v(r)}{2r}=\frac{n-4}{2}\quad\forall r>0.\label{v-eqn}
\end{equation}
By \eqref{h-infty=0} and Theorem \ref{h-steady-soln-asymptotic-behaviour2-thm} there exists a constant $r_0>1$ such that \eqref{h-upper-bd2} and \eqref{q-upper-lower-bd-near-infty} holds.
Let 
\begin{equation}\label{H-defn}
H(r)=\mbox{exp}\,\left(\frac{1}{2}\int_{r_0}^r\left(\frac{n-1}{h(\rho)}+q(r)\right)\,\frac{d\rho}{\rho}\right)\quad\forall r\ge r_0.
\end{equation}
Multiplying \eqref{v-eqn} by $r^{-5/2}H(r)$ and integrating over $(r_0,r)$, $r>r_0$,
\begin{align}\label{v-integral-representation2}
&r^{-5/2}H(r)v(r)=r_0^{-5/2}v(r_0)+\frac{n-4}{2}\int_{r_0}^r\rho^{-5/2}H(\rho)\,d\rho\quad\forall r>r_0\notag\\
\Rightarrow\quad&v(r)=\frac{r_0^{-5/2}v(r_0)}{r^{-5/2}H(r)}+\frac{n-4}{2r^{-5/2}H(r)}\int_{r_0}^r\rho^{-5/2}H(\rho)\,d\rho\quad\forall r>r_0.
\end{align}
By \eqref{h-upper-bd2} and \eqref{q-upper-lower-bd-near-infty},
\begin{align}\label{r-1-H-limit-infty}
&H(r)\ge\mbox{exp}\,\left({\frac{1}{2}\int_{r_0}^r\left(100(n-1)-(3/2)\right)\,\frac{d\rho}{\rho}}\right)\ge \mbox{exp}\,\left(49\int_{r_0}^r\frac{ d\rho}{\rho}\right)
=(r/r_0)^{49}\quad\forall r>r_0\notag\\
\Rightarrow\quad&r^{-5/2}H(r)\ge r_0^{-49}r^{93/2}\to\infty\quad\mbox{ as }r\to\infty\\
\Rightarrow\quad&\int_{r_0}^r\rho^{-5/2}H(\rho)\,d\rho\to\infty\quad\mbox{ as }r\to\infty.\label{integral-to-infty2}
\end{align}
Hence for $n=4$ by \eqref{v-integral-representation2} and \eqref{r-1-H-limit-infty},
\begin{equation*}
\lim_{r\to\infty}v(r)=0.
\end{equation*}
For $n\ge 2$ and $n\ne 4$ by \eqref{r-h-infty-limit}, \eqref{v-integral-representation2}, \eqref{r-1-H-limit-infty},  \eqref{integral-to-infty2},  the l'Hospital rule and Theorem \ref{h-steady-soln-asymptotic-behaviour2-thm},
\begin{align*}
\lim_{r\to\infty}v(r)=&\frac{n-4}{2}\lim_{r\to\infty}\frac{1}{r^{-5/2}H(r)}\int_{r_0}^r\rho^{-5/2}H(\rho)\,d\rho\notag\\
=&\frac{n-4}{2}\lim_{r\to\infty}\frac{r^{-5/2}H(r)}{-(7/2)r^{-7/2}H(r)+r^{-5/2}H(r)[(n-1)(2rh(r))^{-1}+q(r)(2r)^{-1}]}\notag\\
=&\left(\frac{n-4}{n-1}\right)\lim_{r\to\infty}rh(r)=\left(\frac{n-4}{n-1}\right)b_1\quad\forall n\ge 2, n\ne 4
\end{align*}
where $b_1$ is given by \eqref{r-h-infty-limit} and \eqref{h-2nd-order-limit}  follows.

{\hfill$\square$\vspace{6pt}}

\noindent $\underline{\text{Alternate proof of Theorem \ref{steady-soln-asymptotic-hr-behaviour-thm}}}$: Let $w$ be given by \eqref{w1-defn}. Then by Theorem \ref{h-steady-soln-asymptotic-behaviour2-thm} and Lemma \ref{w1-limit-lem},
\begin{align*}
&v(r)=w(r)rh(r)\quad\forall r>0\\
\Rightarrow\quad&\lim_{r\to\infty}v(r)=\lim_{r\to\infty}w(r)\cdot\lim_{r\to\infty}rh(r)=\left(\frac{n-4}{n-1}\right)b_1\quad\forall n\ge 2
\end{align*}
and \eqref{h-2nd-order-limit}  follows.
 
{\hfill$\square$\vspace{6pt}}

\begin{rmk}
By  \eqref{h-2nd-order-limit} of Theorem \ref{steady-soln-asymptotic-hr-behaviour-thm},
\begin{align}\label{h-asymptotic-near-infty}
&\frac{h_r(r)}{h(r)}\approx -\frac{1}{r}-\frac{c_1b_1}{r^2}\quad\mbox{ as }r\to\infty\notag\\
\Rightarrow\quad&\log\left(\frac{h(r)}{h(r_0)}\right)\approx\log\left(\frac{r_0}{r}\right)+c_1b_1\left(\frac{1}{r}-\frac{1}{r_0}\right)\quad\forall 0<r_0<r\quad\mbox{ as }r>r_0\to\infty\notag\\
\Rightarrow\quad&h(r)\approx \frac{C}{r}e^{\frac{c_1b_1}{r}}\quad\mbox{ as }r\to\infty\notag\\
\Rightarrow\quad&\frac{1}{\sqrt{h(r^2)}}\approx\frac{r}{\sqrt{C}}e^{-\frac{c_1b_1}{r^2}}\approx\frac{r}{\sqrt{C}}\left(1-\frac{c_1b_1}{r^2}\right)\quad\mbox{ as }r\to\infty
\end{align}
for some constant $C>0$ depending on $r_0$ where $b_1$ is as given by \eqref{r-h-infty-limit} and
$c_1=\frac{n-4}{n-1}$. Hence by 
\eqref{a-t-relation} and \eqref{h-asymptotic-near-infty},
\begin{equation*}
t=\int_0^{a(t)}\frac{d\rho}{\sqrt{h(\rho^2)}}\approx \frac{1}{\sqrt{C}}\int_{r_0}^{a(t)}\left(\rho-\frac{c_1b_1}{\rho}\right)\,d\rho\approx \frac{1}{2\sqrt{C}}\left(a(t)^2-c_1b_1\log a(t)^2\right)\quad\mbox{ as }a(t)\to\infty
\end{equation*}
where $r_0>0$ is a sufficiently large constant.
Thus $t\to\infty$ as $a(t)\to\infty$ and the metric $g$ given by \eqref{g-rotational-form} or equivalent by \eqref{g=h-form} with $a(t)$ and $h(a(t))$ related by \eqref{a-t-relation} where $h$ is a solution of \eqref{h-steady-soliton-eqn} with $\mu_1<0$ is a complete metric.
\end{rmk}

\section{Asymptotic behaviour of rotationally symmetric expanding gradient Ricci soliton}
\setcounter{equation}{0}
\setcounter{thm}{0}

In this section we will prove the asymptotic behaviour of the solution $h$ of \eqref{h-ode-initial-value-problem} for the case $n\ge 2$ and $\lambda>0$ as $r\to\infty$. We will assume that $\lambda>0$, $\mu_1\in\mathbb{R}\setminus\{0\}$, $n\ge 2$ and $h\in C^2((0,\infty))\cap C^1([0,\infty))$ is a solution of \eqref{h-ode-initial-value-problem}  in this section.

\noindent{\bf Proof of Theorem \ref{expanding-soln-asymptotic-hr-behaviour-thm}}:

\noindent We first observe that  by \eqref{h-derivative-integral-formula0} of
Lemma \ref{h-lower-bd-lem},
\begin{equation}\label{hr-eqn21}
h_r(r)=\frac{n-1}{r}+\lambda+\sqrt{\frac{h(r)}{h(r_1)}}H_1(r_1,r)\le\frac{n-1}{r}+\lambda+\sqrt{\frac{h(r)}{h(r_1)}}H_2(r_1)\quad\forall r>r_1>0
\end{equation}
where 
\begin{equation}\label{H1-defn}
H_1(r_1,r)=h_r(r_1)-\frac{n-1}{r_1}+\frac{(n-1)\sqrt{h(r_1)}}{2}\int_{r_1}^r\frac{h(\rho)+1}{\rho^2\sqrt{h(\rho)}}\,d\rho-\lambda\quad\forall r>r_1>0
\end{equation}
and
\begin{equation}\label{H2-defn}
H_2(r_1)=h_r(r_1)+\frac{(n-1)\sqrt{h(r_1)}}{2}\int_{r_1}^{\infty}\frac{h(\rho)+1}{\rho^2\sqrt{h(\rho)}}\,d\rho-\lambda\quad\forall r_1>0.
\end{equation}
On the other hand by Theorem \ref{h-steady-soln-property-thm1} \eqref{hr-sign} holds. Hence by \eqref{hr-sign},
\begin{equation*}
h_{\infty}:=\lim_{r\to\infty}h(r)\in [0,\infty]\quad\mbox{ exists.}
\end{equation*} 

\noindent{\bf Proof of (i)}: By direct computation \eqref{h-uniform-upper-lower-bd31} is the 
explicit analytic solution of \eqref{h-ode-initial-value-problem}.

\noindent{\bf Proof of (ii)}: Suppose $0<\mu_1<\lambda/n$. Let $\mu_1<\delta<\lambda/n$.
By \eqref{h-ode-initial-value-problem} there exists a constant $s_0>0$ such that
\begin{equation*}
h_r(r)<\delta\quad\forall 0\le r<s_0.
\end{equation*}
Let
\begin{equation*}
R_0=\sup\{r_1>0:h_r(r)<\delta\quad\forall 0\le r<r_1\}.
\end{equation*}
Then $R_0\ge s_0$. Suppose $R_0<\infty$. Then
\begin{align}
&h_r(r)<\delta\quad\forall 0\le r<R_0\quad\mbox{ and }\quad h_r(R_0)=\delta\label{hr-upper-bd5}\\
\Rightarrow\quad&1<h(r)<1+\delta r\quad\forall 0<r\le R_0\quad\mbox{ and }\quad h_{rr}(R_0)\ge 0.\label{h-upper-bd5-hrr-positive}
\end{align}
By \eqref{h-ode-initial-value-problem}, \eqref{hr-upper-bd5} and \eqref{h-upper-bd5-hrr-positive},
\begin{align*}
&2R_0^2h(R_0)h_{rr}(R_0)<(n-1)\left(1+\delta R_0\right)\delta R_0+\delta R_0\left(\delta R_0-\lambda R_0-(n-1)\right)=\delta R_0^2(n\delta-\lambda)<0\notag\\
\Rightarrow\quad&h_{rr}(R_0)<0
\end{align*}
which contradicts \eqref{h-upper-bd5-hrr-positive}. Hence $R_0=\infty$. Thus by letting $\delta\searrow\mu_1$ in \eqref{hr-upper-bd5} and \eqref{h-upper-bd5-hrr-positive},
\begin{equation}\label{h-hr-upper-bd}
0<h_r(r)\le\mu_1\quad\mbox{ and }\quad 1<h(r)\le 1+\mu_1r\quad\forall r>0.
\end{equation}
Thus $h_{\infty}>1$ and \eqref{h-uniform-upper-lower-bd32} follows. 
Suppose 
\begin{equation}\label{h-infty}
h_{\infty}=\infty.
\end{equation}
By \eqref{h-hr-upper-bd} we can divide the proof of (ii) into two cases.

\noindent $\underline{\text{\bf Case 1}}$: $\lim_{r\to\infty}h_r(r)=0$.

\noindent Let $0<\delta_0<\frac{\lambda}{2n}$. Then there exists a constant $r_0>1$ such that 
\begin{align}
&h_r(r)<\delta_0\quad\forall r\ge r_0\label{hr-upper-bd21}\\
\Rightarrow\quad&1<h(r)<h(r_0)+\delta_0 r\quad\forall r\ge r_0.\label{h-upper-bd21}
\end{align}
Hence by \eqref{h-upper-bd21} for any $r_1>r_0$,
\begin{align}\label{h-h-integral}
\frac{\sqrt{h(r_1)}}{2}\int_{r_1}^{\infty}\frac{h(\rho)+1}{\rho^2\sqrt{h(\rho)}}\,d\rho
\le&(h(r_0)+\delta_0r_1)^{1/2}\int_{r_1}^{\infty}\frac{\sqrt{h(\rho)}}{\rho^2}\,d\rho\notag\\
\le&(h(r_0)+\delta_0r_1)^{1/2}(h(r_0)r_1^{-1}+\delta_0)^{1/2}\int_{r_1}^{\infty}\rho^{-3/2}\,d\rho\notag\\
\le&2(h(r_0)+\delta_0r_1)^{1/2}(h(r_0)r_1^{-1}+\delta_0)^{1/2}r_1^{-1/2}.
\end{align}
By \eqref{H2-defn} and \eqref{h-h-integral},
\begin{equation}\label{H2-limit}
\limsup_{r_1\to\infty}H_2(r_1)\le 2(n-1)\delta_0-\lambda<0.
\end{equation}
By \eqref{H2-limit} there exists a constant $r_1>r_0$ such that
\begin{equation}\label{Hr2-negative}
H_2(r_1)<0.
\end{equation}
By \eqref{hr-eqn21}, \eqref{h-infty} and \eqref{Hr2-negative},
\begin{equation*}
\limsup_{r\to\infty}h_r(r)=-\infty
\end{equation*}
which contradicts \eqref{hr-sign}. Hence case 1 does not hold.

\noindent $\underline{\text{\bf Case 2}}$: There exists a constant $\delta_1>0$ and a sequence $\{s_i\}_{i=1}^{\infty}\subset (1,\infty)$, $s_i\to\infty$ as $i\to\infty$, such that 
\begin{equation}\label{hr-positive6}
h_r(s_i)\ge\delta_1\quad\forall i\in\Z^+.  
\end{equation}

\noindent By \eqref{h-hr-upper-bd} and \eqref{hr-positive6} the sequence $\{s_i\}_{i=1}^{\infty}$ has a subsequence which we may assume without loss of generality to be the sequence itself such that
\begin{equation}\label{hr-sequence-limit-positive}
h_r(s_i)\to\delta_2\quad\mbox{ as }i\to\infty  
\end{equation}
for some constant $\delta_2\in (0,\mu_1]$. By \eqref{h-hr-upper-bd},
\begin{align}\label{h-integral-limit}
\int_{s_i}^{\infty}\frac{h(\rho)+1}{\rho^2\sqrt{h(\rho)}}\,d\rho
\le&2\int_{s_i}^{\infty}\frac{\sqrt{h(\rho)}}{\rho^2}\,d\rho
\le 2(1+\mu_1)^{1/2}\int_{s_i}^{\infty}\rho^{-3/2}\,d\rho
\le 4(1+\mu_1)^{1/2}s_i^{-1/2}\quad\forall i\in\Z^+\notag\\
\to&0\quad\mbox{ as }i\to\infty.
\end{align}
By \eqref{h-infty},   \eqref{hr-sequence-limit-positive}, \eqref{h-integral-limit} and the l'Hospital rule,
\begin{align}\label{h-integral-ratio-limit}
\lim_{i\to\infty}\frac{\sqrt{h(s_i)}}{2}\int_{s_i}^{\infty}\frac{h(\rho)+1}{\rho^2\sqrt{h(\rho)}}\,d\rho
=&\lim_{i\to\infty}\frac{\int_{s_i}^{\infty}\frac{h(\rho)+1}{2\rho^2\sqrt{h(\rho)}}\,d\rho}{h(s_i)^{-1/2}}
=\lim_{i\to\infty}\frac{\frac{h(s_i)+1}{s_i^2\sqrt{h(s_i)}}}{h(s_i)^{-3/2}h_r(s_i)}\notag\\
=&\delta_2^{-1}\lim_{i\to\infty}\frac{h(s_i)}{s_i}\cdot\lim_{i\to\infty}\frac{h(s_i)+1}{s_i}\notag\\
=&\delta_2^{-1}\lim_{i\to\infty}h_r(s_i)\cdot\lim_{i\to\infty}h_r(s_i)\notag\\
=&\delta_2.
\end{align}
Putting $r_1=s_i$, $r=\infty$, in \eqref{H1-defn} and letting $i\to\infty$, by \eqref{hr-sequence-limit-positive} and \eqref{h-integral-ratio-limit},
\begin{equation*}
\lim_{i\to\infty}H_1(s_i,\infty)=\delta_2+(n-1)\delta_2-\lambda=n\delta_2-\lambda\le n\mu_1-\lambda<0.
\end{equation*}
Hence there exists $i_0\in\Z^+$ such that
\begin{equation}\label{Hsi-negative}
H_1(s_{i_0},\infty)<0.
\end{equation}
Putting $r_1=s_{i_0}$ in \eqref{hr-eqn21} and letting $r\to\infty$, by \eqref{hr-eqn21}, \eqref{h-infty} and \eqref{Hsi-negative},
\begin{equation*}
\limsup_{r\to\infty}h_r(r)=-\infty
\end{equation*}
which contradicts \eqref{hr-sign}. Hence case 2 does not hold. Thus \eqref{h-infty} does not hold and 
$h_{\infty}<\infty$. Putting $r_1=1$ and letting $r\to\infty$ in \eqref{h-derivative-integral-formula0},
by \eqref{h-hr-upper-bd},
\begin{equation}\label{hr-limit7}
\lim_{r\to\infty}h_r(r)=\lambda+\sqrt{\frac{h_{\infty}}{h(1)}}\left(h_r(1)-(n-1)-\lambda\right)+(n-1)\frac{\sqrt{h_{\infty}}}{2}\int_1^{\infty}\frac{h(\rho)+1}{\rho^2\sqrt{h(\rho)}}\,d\rho\in [0,\mu_1].
\end{equation}
Since $h_{\infty}<\infty$, by \eqref{hr-limit7},
\begin{equation*}
\lim_{r\to\infty}h_r(r)=0
\end{equation*}
and (ii) follows.

\noindent{\bf Proof of (iii)}: Suppose $\mu_1>\lambda/n$. Let $\mu_1>\delta>\lambda/n$.
By \eqref{h-ode-initial-value-problem} there exists a constant $s_0'>0$ such that
\begin{equation*}
h_r(r)>\delta\quad\forall 0\le r<s_0'.
\end{equation*}
Let
\begin{equation*}
R_0=\sup\{r_1>0:h_r(r)>\delta\quad\forall 0\le r<r_1\}.
\end{equation*}
Then $R_0\ge s_0'$. Suppose $R_0<\infty$. Then
\begin{align}
&h_r(r)>\delta\quad\forall 0\le r<R_0\quad\mbox{ and }\quad h_r(R_0)=\delta\label{hr-lower-bd5}\\
\Rightarrow\quad&h(r)>1+\delta r\quad\forall 0<r\le R_0\quad\mbox{ and }\quad h_{rr}(R_0)\le 0.\label{h-lower-bd5-hrr-negative}
\end{align}
By \eqref{h-ode-initial-value-problem}, \eqref{hr-lower-bd5} and \eqref{h-lower-bd5-hrr-negative},
\begin{align*}
&2R_0^2h(R_0)h_{rr}(R_0)>(n-1)\left(1+\delta R_0\right)\delta R_0+\delta R_0\left(\delta R_0-\lambda R_0-(n-1)\right)=\delta R_0^2(n\delta-\lambda)>0\notag\\
\Rightarrow\quad&h_{rr}(R_0)>0
\end{align*}
which contradicts \eqref{h-lower-bd5-hrr-negative}. Hence $R_0=\infty$. Thus by letting $\delta\nearrow\mu_1$ in \eqref{hr-lower-bd5} and \eqref{h-lower-bd5-hrr-negative},
\begin{align}
&h_r(r)\ge\mu_1\quad\mbox{ and }\quad h(r)\ge 1+\mu_1r\quad\forall r>0\label{h-hr-lower-bd}\\
\Rightarrow\quad&h_{\infty}=\lim_{r\to\infty}h(r)=\infty.\notag
\end{align}
Hence \eqref{h-uniform-upper-lower-bd33} holds. By \eqref{h-hr-lower-bd},
\begin{equation}\label{h-integral-lower-bd7}
\sqrt{h(r_1)}\int_{r_1}^{\infty}\frac{h(\rho)+1}{2\rho^2\sqrt{h(\rho)}}\,d\rho
\ge\frac{(1+\mu_1r_1)^{1/2}\sqrt{\mu_1}}{2}\int_{r_1}^{\infty}\rho^{-3/2}\,d\rho
\ge\mu_1\quad\forall r_1>0.
\end{equation}
By  \eqref{H1-defn}, \eqref{h-hr-lower-bd} and \eqref{h-integral-lower-bd7},
\begin{equation*}
\liminf_{r_1\to\infty}H_1(r_1,\infty)\ge\mu_1+(n-1)\mu_1-\lambda=n\mu_1-\lambda>0.
\end{equation*}
Hence there exists a constant $r_1>0$ such that
\begin{equation*}
H_1(r_1,\infty)>0.
\end{equation*}
Thus there exists a constant 
$r_0'>r_1$ such that
\begin{equation}\label{H1-strictly-positive}
H_1(r_1,r)\ge H_1(r_1,r_0')>0\quad\forall r\ge r_0'.
\end{equation}
Letting $r\to\infty$ in \eqref{hr-eqn21}, by \eqref{H1-strictly-positive},
\begin{equation*}
\lim_{r\to\infty}h_r(r)=\infty
\end{equation*}
and (iii) follows.

\noindent{\bf Proof of (iv)}: Suppose $h_{\infty}=0$. Then by \eqref{hr-eqn21},
\begin{align*}\label{hr-eqn23}
&h_r(r)\ge\frac{n-1}{r}+\lambda+\sqrt{\frac{h(r)}{h(r_1)}}\left(h_r(r_1)-\frac{n-1}{r_1}-\lambda\right)\quad\forall r>r_1>0\\
\Rightarrow\quad&\liminf_{r\to\infty}h_r(r)\ge\lambda>0
\end{align*}
which contradicts \eqref{hr-sign}. Hence by \eqref{hr-sign} $h_{\infty}\in (0,1)$. Thus,
\begin{equation}\label{h-integral-bd15}
\int_1^{\infty}\frac{h(\rho)+1}{\rho^2\sqrt{h(\rho)}}\,d\rho<\infty.
\end{equation}
Putting $r_1=1$ and letting $r\to\infty$ in \eqref{h-derivative-integral-formula0}, by \eqref{hr-sign} and \eqref{h-integral-bd15},
\begin{equation*}
h_r^{\infty}:=\lim_{r\to\infty}h_r(r)=\lambda+\sqrt{\frac{h_{\infty}}{h(1)}}\left(h_r(1)-(n-1)-\lambda\right)+(n-1)\frac{\sqrt{h_{\infty}}}{2}\int_1^{\infty}\frac{h(\rho)+1}{\rho^2\sqrt{h(\rho)}}\,d\rho\in (-\infty,0].
\end{equation*}
If $h_r^{\infty}<0$, then there exists a constant $r_0>0$  such that
\begin{align*}
&h_r(r)\le h_r^{\infty}/2<0\quad\forall r\ge r_0\notag\\
\Rightarrow\quad&h(r)\le h(r_0)+(r-r_0)(h_r^{\infty}/2)\le 1+(r-r_0)(h_r^{\infty}/2)<0\quad\forall r>r_0+(2/|h_r^{\infty}|)
\end{align*}
which  contradicts \eqref{hr-sign}. Hence
\begin{equation*}
h_r^{\infty}:=\lim_{r\to\infty}h_r(r)=0
\end{equation*}
and (iv) follows.

{\hfill$\square$\vspace{6pt}}

Finally by an argument similar to the proof of (iii) of Theorem \ref{expanding-soln-asymptotic-hr-behaviour-thm} we get that Theorem \ref{steady-soln-asymptotic-behaviour-for-hr-positive-thm} holds.

\end{document}